\documentclass[12pt,leqno]{amsart} 
\usepackage{amsmath,amstext,amsthm,amssymb,amsxtra}
\usepackage[backrefs]{amsrefs}
\allowdisplaybreaks
 
\theoremstyle{plain} 
\newtheorem{lemma}[equation]{Lemma} 
\newtheorem{proposition}[equation]{Proposition} 
\newtheorem{theorem}[equation]{Theorem} 
\newtheorem{corollary}[equation]{Corollary} 
\newtheorem*{maintheorem}{Main Theorem}
\newtheorem*{weakfactor}{Weak Factorization Theorem}
\newtheorem{pCZ}[equation]{Product Calder\'on--Zygmund Kernels}
\newtheorem{extended}[equation]{Extension of Main Theorem}

\theoremstyle{definition}

\theoremstyle{remark}
\newtheorem*{remark}{Remark}
\newtheorem*{Acknowledgment}{Acknowledgment}

\numberwithin{equation}{section}

\def\norm#1.#2.{\lVert#1\rVert_{#2}}
\def\Norm#1.#2.{\bigl\lVert#1\bigr\rVert_{#2}}
\def\NOrm#1.#2.{\Bigl\lVert#1\Bigr\rVert_{#2}}
\def\NORm#1.#2.{\biggl\lVert#1\biggr\rVert_{#2}}
\def\NORM#1.#2.{\Biggl\lVert#1\Biggr\rVert_{#2}}

\def\ip#1,#2,{\langle #1,#2\rangle}
\def\Ip#1,#2,{\bigl\langle#1,#2\bigr\rangle}
\def\IP#1,#2,{\Bigl\langle#1,#2\Bigr\rangle}

\def\mid{\,:\,}

\def\abs#1{\lvert#1\rvert}

\def\ABs#1{\Bigl\lvert#1\Bigr\rvert}

\def\XXint#1#2#3{{\setbox0=\hbox{$#1{#2#3}{\int}$}
     \vcenter{\hbox{$#2#3$}}\kern-.5\wd0}}

\def\eqdef{\stackrel{\mathrm{def}}{{}={}}}

\DeclareFontFamily{U}{wncy}{}
\DeclareFontShape{U}{wncy}{m}{n}{<->wncyr10}{}
\DeclareSymbolFont{mcy}{U}{wncy}{m}{n}
\DeclareMathSymbol{\Sh}{\mathord}{mcy}{"58}  

\begin{document}
\title {Multiparameter Riesz Commutators}
\author[M.T. Lacey\and S. Petermichl\and J.C. Pipher\and B.D. Wick]{Michael T. Lacey$^1$\and Stefanie Petermichl$^2$ \and Jill C. Pipher$^3$ \and Brett D. Wick$^4$}

\address{Michael T. Lacey\\
School of Mathematics\\
Georgia Institute of Technology\\
Atlanta, GA 30332\\
}
\email{lacey@math.gatech.edu}
\thanks{$1.$ Research supported in part by a National Science Foundation Grant. The author is a Guggenheim Fellow.}

\address{Stefanie Petermichl\\ Department of Mathematics\\ University of Texas at Austin\\ Austin, TX  78712\\}
\email{stefanie@math.utexas.edu}
\thanks{$2.$ Research supported in part by a National Science Foundation Grant.}

\address{Jill C. Pipher\\ Department of Mathematics\\ Brown University\\ Providence, RI 02912\\}
\email{jpipher@math.brown.edu}
\thanks{$3.$ Research supported in part by a National Science Foundation Grant.}

\address{Brett D. Wick\\ Department of Mathematics\\ Vanderbilt University\\ Nashville, TN 37240\\}
\email{brett.d.wick@vanderbilt.edu}
\thanks{$4.$ Research supported in part by a National Science Foundation RTG Grant to Vanderbilt University.}

\begin{abstract}
It is shown that product BMO of S.-Y.\thinspace A.\thinspace Chang and R. Fefferman, 
defined on the space $\mathbb R^{d_1}\otimes \cdots 
\otimes \mathbb R ^{d_t}$, 
can be characterized by the multiparameter commutators of Riesz transforms.
This extends a classical 
one-parameter  result of R.~Coifman, R.~Rochberg, and G.~Weiss  
\cite{MR54:843}, and at the same time extends the work of 
M.~Lacey and S.~Ferguson \cite{sarahlacey} and M.~Lacey and E.~Terwilleger
\cite{math.CA/0310348}, on multiparameter commutators 
with Hilbert transforms. 
\end{abstract}

\maketitle

\section{Introduction}

 In one parameter, a classical result of Coifman, Rochberg 
 and Weiss \cite{MR54:843}, in turn an extension of the result of Nehari 
 \cite{nehari}, 
 shows that a function in the Hardy space  $ H^1$ on the 
 ball can be weakly factored as a sum of products of functions in $ H^2$  on the ball.  Recently, 
 Ferguson and Lacey 
\cite{sarahlacey} and  Lacey and Terwilleger \cite{math.CA/0310348}  proved the corresponding weak 
factorization for $ H^1$ of the polydisc.  In this paper, we prove the real variable 
generalization of these two sets of results.
 
 Let $ \operatorname M _b \varphi \eqdef b \cdot \varphi $ be the operator of pointwise multiplication 
 by a function $ b$.  
 For Schwartz functions $ f$ on  $ \mathbb R^{d}$, let $ \operatorname R_j f$ denote the 
$ j$th Riesz transform of $ f$, for $ 1\le j\le d$. From time to time, we will 
use the notation $ \operatorname R_0$ for the identity operator. 

We are concerned with product spaces $ \mathbb R ^{\vec d}=\mathbb R^{d_1}\otimes \cdots \otimes \mathbb R ^{d_t}$ for vectors $ \vec d=(d_1,\ldots , d_t)\in
\mathbb N ^{t}$. 
For Schwartz functions $b,f$ on $ \mathbb R ^{\vec d}$, 
and for a vector $\vec \jmath =(j_1,\ldots, j_t)$ with $1\leq j_s\leq d_s$ for $s=1,\ldots,
t$ we consider the family of commutators
\begin{equation} \label{e.commutators}
\operatorname C _{\vec \jmath }(b, f)
\eqdef [\cdots[[\operatorname M_b, \operatorname R _{1,\, j_1}], 
\operatorname R _{2,\, j_2}],\cdots], \operatorname R _{t,j_t} ]f
\end{equation}
where  $ \operatorname R _{s,\, j}$ denotes the $ j$th Riesz transform
acting on $ \mathbb R ^{d_s}$.

\begin{maintheorem}  We have the estimates below, valid for $ 1<p<\infty $. 
\begin{equation}\label{e.BMO}
\sup _{\vec \jmath }\norm \operatorname C_{\vec \jmath }(b, \cdot  ).p \to p. 
\simeq   \norm b.\textup{BMO}.\,.
\end{equation}
By $ \textup{BMO}$, we mean Chang--Fefferman $ \textup{BMO}$. 
\end{maintheorem}

To establish this result, we find it necessary to  prove an extended 
version of this Theorem, proving the equivalence of norms not only for the 
Riesz transforms, but also a class of singular integral operators whose symbols  
are supported on cones.  In this, and other ways, our methods shed 
new light on issues related to commutators even in the one parameter case.

It is well known that the result above has an equivalent formulation in 
terms of \emph{weak factorization} of Hardy space; indeed, this equivalence is 
important to the proof of the Theorem. For $\vec \jmath $ a 
vector with $1\leq j_s\leq d_s$, and $s=1,\ldots, t$, let $\operatorname \Pi_{\vec \jmath }$ 
be the bilinear operator defined by the following equation
$$
\langle C_{\vec \jmath }(b,f),g\rangle\eqdef\langle b, \Pi_{\vec \jmath }(f,g)\rangle.
$$
One can express $\operatorname \Pi_{\vec \jmath }$ as a linear combination of 
products of iterates of Riesz transforms, $\operatorname R_{s,j_s}$, applied
to the $f$ and $g$.  It follows immediately by duality from the Main Theorem 
that for sequences $f_k^{\vec \jmath },g_k^{\vec \jmath }\in L^2(\mathbb R^{\vec d})$ with 
$\sum_{\vec \jmath }\sum_{k=1}^\infty\norm f_k^{\vec \jmath }.2.\norm g_k^{\vec \jmath }.2.<\infty$ 
we have
$$
\sum_{\vec \jmath }\sum_{k=1}^\infty\operatorname \Pi_{\vec \jmath }
(f_k^{\vec \jmath },g_k^{\vec \jmath })\in H^1(\mathbb R^{\vec d}).
$$

With this observation, we define 
\begin{equation}\label{e.LotimesL}
L^2(\mathbb R^{\vec d})\widehat\odot L^2(\mathbb R^{\vec d})\eqdef \Bigl\{
f\in L^1(\mathbb R^{\vec d}) : f=\sum_{\vec \jmath }\sum_{k=1}^\infty
\operatorname\Pi_{\vec \jmath }(f_k^{\vec \jmath },g_k^{\vec \jmath })\Bigr\}\,.
\end{equation}
This is the projective  product given by 
$$
\norm f.L^2(\mathbb R^{\vec d})\widehat\odot L^2(\mathbb R^{\vec d}).\eqdef
\inf\Bigl\{\sum_{\vec \jmath }\sum_k\norm f_k^{\vec \jmath }.2.
\norm g_k^{\vec \jmath }.2.\Bigr\}
$$
where the infimum is taken over all  decompositions of $ f$ as in 
(\ref{e.LotimesL}).  This definition has an obvious extension to 
$ L^p \widehat\odot L ^{p'} $, for $ 1\le p,p'<\infty $. 
We  have the following corollary.

\begin{weakfactor}
For any $ 1<p<\infty $, let $ p'=p/(p-1)$ be the conjugate index.  We have 
$ H^1(\mathbb R^{\vec d})= L^p(\mathbb R^{\vec d})\widehat\odot L^{p'}(\mathbb R^{\vec d})$. 
Namely, for any $f\in H^1(\mathbb R^{\vec d})$ 
there exist sequences $f_k^{\vec \jmath }\in L^p(\mathbb R^{\vec d})$ 
and $ g_k^{\vec \jmath }\in L ^{p'}(\mathbb R^{\vec d})$
such that
\begin{gather*}
f=\sum_{\vec \jmath } \sum_{k=1}^{\infty}\operatorname
\Pi_{\vec \jmath }(f_k^{\vec \jmath },g_k^{\vec \jmath })\,,
\\
\norm f.H^1. \simeq \sum_{\vec \jmath }
\sum _{k} \norm f_k^{\vec \jmath }.p. \norm g_k^{\vec \jmath }.p'. \,.
\end{gather*}
\end{weakfactor}


The result of  Coifman, Rochberg 
and Weiss \cite{MR54:843} has found a number of 
further applications.  The original paper includes a weak factorization result for 
certain Bergman spaces, and there is the striking application to the theory of 
compensated compactness \cite{MR1225511}. We anticipate that some of these 
applications persist into the higher parameter setting of this paper, but we do not 
purse these points in this paper. 

The proofs given here are rather different from that of Coifman, Rochberg 
and Weiss \cite{MR54:843}.  
Their proof of the upper bound on commutator norms relies upon a sharp 
function inequality, a method of proof that is quite powerful in the one 
parameter setting.  This method admits only a weak extension to the 
higher parameter setting;  instead, our proof of the upper bound, 
namely Theorem~\ref{t.upper}, 
follows from the decomposition of the commutators into a sum of simpler terms. 
These terms are  paraproducts, composed on either side 
by Calder\'on--Zygmund operators.  This method has been used in different settings, 
such as Petermichl  \cite{MR1756958}, and Lacey \cite{math.CA/0502336}. 
Our formalization of this method in this paper could lead to further applications of this
method. 

The paraproducts that arise are of multiparameter form.  The specific 
result needed
is Theorem~\ref{t.MainParaproducts} below. 
This result is due to Journ\'e \cite{MR88d:42028}; 
more recent 
discussions of paraproducts  are in  \cites{camil1,camil2,math.CA/0502334}.

For the lower bound, namely Theorem~\ref{t.lower},
we use the strategy of Ferguson--Lacey \cite{sarahlacey} and 
Lacey--Terwilleger \cite{math.CA/0310348}. One inducts on parameters, 
using a boot-strapping argument, 
and the Journ\'e Lemma \cite{MR87g:42028}.  
However, to implement this strategy, we have to prove a second version 
of the Main Theorem, one in which the Riesz transforms are replaced by the 
a family of Calder\'on--Zygmund operators whose symbols are adapted to cones. 
These kernels are described in the text preceding (\ref{e.ConeNorm}), and 
Theorem~\ref{t.extended} is the extended version of our Main Theorem.  

To start the induction, in the case of Riesz transforms, we can of course 
use the Coifman, Rochberg, and Weiss result.  But for the cones, we appeal 
to the results of   Uchiyama \cite{MR0467384} and 
Song-Ying Li \cite{MR1373281} which are deep extensions of the work of 
Coifman, Rochberg and Weiss. 

In \S~\ref{s.wavelet}---\ref{s.para} we recall different aspects of the 
multiparameter theory in forms required for this investigation. 
\S\ref{s.upper} introduces the cone operators, and this establishes the upper 
bound on commutator norms.  The initial stages of the 
lower bound on commutator norms is proved in  
\S\ref{s.lower}.  The more refined bootstrapping argument occupies 
\S\ref{s.boot}. 

\begin{Acknowledgment}
The authors benefited from the Banff Research Station, through the 
Research in Teams program, and a very pleasant stay at the 
Universit\'e de Bordeaux.  The authors also enjoyed a stay at Texas A\&M University.
\end{Acknowledgment}

\section{Wavelets in Several Dimensions} \label{s.wavelet}

This discussion is initially restricted to a one parameter setting. 
We will use dilation and translation operators on $ \mathbb R ^{d}$
\begin{align}\label{e.trans}
\operatorname {Tr} _{y} f(x) &\eqdef f(x-y)\,,\quad y\in \mathbb R ^{d}, 
\\
\label{e.dil}
\operatorname {Dil} _{a} ^{(p)} f(x) & \eqdef  a ^{-d/p} f(x/a)\,, 
\quad a>0\,,\ 0<p\le \infty \,.
\end{align}
These will also be applied to sets, in an obvious fashion, in the case of $ p= \infty $.  

By the \emph{ ($ d$ dimensional) dyadic grid} in $ \mathbb R ^{d}$ we mean the collection of cubes 
\begin{equation*}
\mathcal D_d \eqdef \bigl\{ j 2 ^{k}+[0,2^k)^ d\mid j \in \mathbb Z ^{d}\,,\ 
k\in \mathbb Z \bigr\}.
\end{equation*}

 \emph{Wavelets} arise from a  mean zero Schwartz function $ w$, 
a \emph{scaling function}, expressible
 as $ W(x)-2W (2x)$, for a \emph{father wavelet $ W$}. The principle 
requirement is that the functions 
 $ \{\operatorname{Tr}_{c(I)}\operatorname {Dil} ^{(2)} _I w \mid I\in \mathcal D_1\}$ form 
 an orthonormal basis for $ L ^{2} (\mathbb R )$.  Except for the fact that 
 it is not smooth, $  h= -\mathbf 1 _{(0,1/2)}+\mathbf 1 _{(1/2,1)}$ 
 is a scaling function, with father wavelet $ \mathbf 1 _{(0,1)}$. 
 This generates the Haar basis for $ L^2 (\mathbb R )$. 
 
  For $ \varepsilon \in \{0,1\}$, set $ w^0=w$ and $ w ^{1}=W$, 
the superscript $ {}^0$  denoting that `the function has mean $ 0$,' 
while a superscript $ {}^1$ denotes that 
`the function is an  $ L^2$ normalized indicator function.'
In one dimension, for an interval $ I$, set 
\begin{equation*}
w ^{\varepsilon } _{I} \eqdef \operatorname {Tr} _{c(I)} \operatorname {Dil} _{\abs{ I}} 
^{(2)} w ^{\varepsilon }\,.
\end{equation*}
The father wavelet is of some convenience to us, as we have the useful 
facts, valid on the interval $J$\footnote{Technically, these results are only true for \emph{multiresolution 
analysis (MRA)} wavelets.  Both the Haar and Meyer wavelets are MRA wavelets. } 
\begin{equation}\label{e.FATHER}
\sum _{I\supsetneq J} \ip f,w_I, w_I = \ip f, w ^{1}_J, w ^{1}_J\,,
\end{equation}

 We will use the \emph{Meyer wavelet} in later sections of the paper. 
 This wavelet, found by Y.~Meyer \cites {MR1009177,MR1085487},  
arises from a Schwartz scaling function $ w $, with $ \widehat w  $ 
supported on $  1/3\le \abs{ \xi } \le 8/3$.  Indeed, $ \widehat w $ is 
identically equal to 1 on the intervals $ 1\le \abs{ \xi } \le 2 $. 
The `father wavelet' 
$ W$ is a Schwartz function with $ \widehat W $ supported on $ \abs{ \xi }<2$, 
so that  $ w(x)=W(x)-2 W(2x)$.  
One of the reasons this is such a useful wavelet for us is the fact below 
which is exploited several times.  
\begin{equation}\label{e.Fourier}
8\abs{ I}< \abs{ I'} \quad \textup{implies} \quad \textup{$\widehat{  w_I \cdot w _{I'}}$ 
is supported on $ (4 \abs{ I}) ^{-1} < \abs{ \xi }< 3 \abs{ I} ^{-1}  $.}
\end{equation}
%

\bigskip

Let $\textup {Sig}_{d} \eqdef \{0,1\}^d-\{\vec 1\}$, which we refer to as 
\emph{signatures}.
In $ d$ dimensions, for a cube $Q$ with side $\abs{I}$, i.e.,  $Q=I_1\times \cdots \times I_d$, and a choice of $ \varepsilon\in\textup{Sig}_d$, set 
\begin{equation*}
w ^{\varepsilon} _Q (x_1,\dotsc,x_d)\eqdef \prod _{j=1} ^{d} w _{I_j} ^{\varepsilon _j} (x_j).
\end{equation*}
It is then the case that the collection of functions 
\begin{equation*}
\operatorname {Wavelet} _{\mathcal D_d} \eqdef \{ w _{Q} ^{\varepsilon } 
\mid Q\in \mathcal D_d\,, \ \varepsilon\in\textup{Sig}_d\}
\end{equation*}
form a wavelet basis for $ L^p (\mathbb R ^{d})$ for any choice of  $ d$ dimensional 
dyadic grid $ \mathcal D_d$. 
Here, we are using the notation $ \vec 1=(1,\dotsc,1)$.  While we exclude the 
superscript $ {} ^{\vec 1}$ here, it plays a role in the theory of 
paraproducts.

We will use these bases in the tensor product setting. Thus, for 
a vector $ \vec d=(d_1,\dotsc,d_t)$, and $ 1\le s\le t$, 
let $ \mathcal D _{d_s}$ be a choice of $ d_s$ dimensional dyadic grid, and let 
\begin{equation*} 
\mathcal D _{\vec d}=\otimes _{s=1} ^{t} \mathcal D _{d_s}\,. 
\end{equation*}
Also, let $\textup {Sig}_{\vec d} \eqdef \{\vec\varepsilon=(\varepsilon_1,\ldots,\varepsilon_t):\varepsilon_s\in\textup{Sig}_{d_s}\}$.
Note that each $\varepsilon_s$ is a vector, and so $\vec\varepsilon$ is a 
`vector of vectors'.  For a rectangle $ R=Q_1\times \cdots  \times Q_t$, being a product of cubes 
of possibly different dimensions, and a choice of vectors
$ \vec\varepsilon\in\textup {Sig}_{\vec d}$
set 
\begin{equation*}
w _{R} ^{\vec \varepsilon }(x_1,\dotsc,x_t)=\prod _{s=1}^t w _{Q_s} ^{\varepsilon _s}(x_s).
\end{equation*}
These are the appropriate functions and bases to analyze multiparameter 
paraproducts and commutators.  

Let
\begin{equation*}
\operatorname {Wavelet} _{\mathcal D _{\vec d}} \eqdef \bigl\{ w _{R} ^{\vec \varepsilon }\mid 
R\in \mathcal D _{\vec d}\,, \ \vec\varepsilon\in\textup {Sig}_{\vec d}\bigr\}\,. 
\end{equation*}
This is a basis in $ L^p(\mathbb R ^{\vec d})$, where we will use the notation 
\begin{equation*}
\mathbb R ^{\vec d} \eqdef \mathbb R ^{d_1}\otimes \cdots \otimes \mathbb R ^{d_t}
\end{equation*}
to emphasize that we are in a tensor product setting. 


\section{Chang--Fefferman $ \textup{BMO}$} \label{s.cfBMO}

We describe the elements of product Hardy space theory, 
as developed by S.-Y.~Chang and R.~Fefferman \cites{cf1,cf2,
MR90e:42030,MR86f:32004,MR81c:32016} as well as Journ\'e \cites{MR87g:42028,MR88d:42028}.  
By this, we mean 
the Hardy spaces associated with domains like $\otimes_{s=1}^t \mathbb R^{d_s}$.

\begin{remark}\label{r.H(Rd)}
The (real) Hardy space $ H^1 (\mathbb R ^{d})$ typically denotes the class of functions 
with the norm 
\begin{equation*}
\sum _{j=0} ^{d} \norm \operatorname R_j f.1.
\end{equation*}
where $ \operatorname R_j$ denotes the $ j$th Riesz transform. 
Here and below we adopt the convention that $ \operatorname R_0$, 
the $ 0$th Riesz transform, is the identity. 
This space is 
invariant under the one parameter family of isotropic dilations, while   $ H^1(\mathbb R^{\vec d}) $ 
is invariant under dilations of each coordinate separately. That is, it is invariant 
under a $t$ parameter family of dilations, hence the terminology  `multiparameter' 
theory.
\end{remark}

As before, the space $H^1(\mathbb R^{\vec d}) $ has a variety of equivalent norms, 
in terms of square functions, maximal functions and Riesz transforms.  
For our discussion of paraproducts, it is appropriate to make some definitions of translation 
and dilation operators which extend the definitions in (\ref{e.trans})---(\ref{e.dil}). 
(Indeed, here we are adopting broader notation than we really need, in anticipation 
of a discussion of multiparameter paraproducts.) 
Define 
\begin{align}\label{e.translated}
\operatorname {Tr} _{y}f(x) &\eqdef f(x-y),\qquad y\in \mathbb R ^{\vec d}\,,
\\ \label{e.dilated}
\operatorname {Dil} _{a_1,\dotsc,a_t}^p f(x_1,\dotsc,x_t) 
&\eqdef \prod_{s=1}^ta_s^{-{d_s}/p}f(x_1/a_1,\dotsc,x_t/a_t)\,,\qquad a_1,\dotsc,a_t>0\,,
\\ \label{e.transDild}
\operatorname {Dil} _{R} ^{p} &\eqdef \operatorname {Tr} _{c(R)} \operatorname {Dil} 
_{\abs{ Q_1},\dotsc , \abs{Q_d}} ^{p}\,.
\end{align}
In the last definition $ R=Q_1\times\cdots\times Q_t$ is a rectangle, each $Q_s$ 
is a cube and the dilation 
incorporates the locations and scales associated with $ R$.  $ c (R)$ is the center of $ R$.

For a non-negative smooth bump function $ \varphi ^{1}$ in $\mathbb R^{\vec d}$ with $\int \varphi ^{1}\; dx=1$, 
define the (strong) maximal function by 
\begin{equation*}
\operatorname M f(x)\eqdef\sup _{R\in \mathcal D_{\vec d}} 
\operatorname {Dil} _{R} ^{2} \varphi^1 (x) \ip f,\operatorname {Dil} _{R} ^{2}\varphi ^{1},  .
\end{equation*}

For $s=1,\ldots, t$, choose radial bump functions $\varphi^{0}_s$ on $\mathbb R^{d_s}$ with $ \int _{\mathbb R^{d_s} } \varphi_s^{0} \; dx_s=0$ and
\begin{equation*}
\sup _{\xi}\int_{0}^\infty\abs{\widehat{\varphi^0_s}(t\xi)}^2\frac{dt}{t}<\infty\,.
\end{equation*}
Then, fix $\varphi ^{0}$ so that 
\begin{equation*}
\varphi ^{0}(x_1,\dotsc,x_t)\eqdef\prod _{s=1} ^{t}\varphi^{0}_s (x_s).
\end{equation*}
As an analog of the Littlewood--Paley square function, set 
\begin{equation*}
\operatorname S f(x)\eqdef
\Bigl[\sum _{R\in \mathcal D_{\vec d}} 
[\operatorname {Dil} _{R} ^{2} \varphi^0 (x)] ^{2}
\abs{ \ip f,\operatorname {Dil} _{R} ^{2} \varphi ^{0},  } ^2\Bigr] ^{1/2} .
\end{equation*}

\begin{theorem}[Equivalent forms of $ H^1$ norm] 
\label{t.H^1equiv}
All of the norms below are equivalent, and can be used as 
a definition of $H ^{1} (\mathbb R^{\vec d})$. 
\begin{gather*}
\norm \operatorname M f.1.
\simeq 
\norm f.1.+
\norm \operatorname S f.1.
\simeq 
\sum_{\substack{\vec 0\le \vec \jmath \le \vec d }}   
\NOrm \prod _{s=1} ^{t} \operatorname R_{s,j_s} f.1.\,.
\end{gather*}
 $ \operatorname R_{s,j_s}$ is  the Riesz transform 
computed in the $j_s$th direction of the $s$th variable, and the 
$ 0$th Riesz transform is the identity operator. 
\end{theorem}

\subsection{$ \textup{BMO} (\mathbb R^{\vec d})$}

The dual of the real Hardy space is $ H^1(\mathbb R^{\vec d})^\ast=\text{BMO}(\mathbb R^{\vec d})$,
the $t$--fold product $\text{BMO}$ space. It is a Theorem of S.-Y.~Chang and R.~Fefferman 
\cite{cf2} that this space has a characterization 
in terms of a product Carleson measure.

Define 
\begin{equation} \label{e.BMOdef}
\lVert b\rVert_{\text{BMO} ( \mathbb R^{\vec d})}\eqdef
\sup_{ U\subset \mathbb R^{\vec d}} \Bigl[\abs{   U}^{-1} 
\sum_{R\subset U}\sum_{\vec{\varepsilon}\in\textup {Sig}_{\vec d}}\abs{\ip b,w_R^{\vec\varepsilon},}^2\Bigr] ^{1/2} .
\end{equation}
Here the supremum is taken over all open subsets $U\subset \mathbb R^{\vec d}$ with finite measure, 
and we use a wavelet basis $ w_R ^{\vec \varepsilon }$.

\begin{theorem}[Chang--Fefferman $ \textup{BMO}$]\label{t.changfefferman}
We have the equivalence of norms 
\begin{equation*}
\norm f. (H^1 (\mathbb R ^{\vec d})) ^{\ast}. 
\approx \norm f. \textup{BMO}(\mathbb R ^{\vec d}). .
\end{equation*}
That is, $ \textup{BMO}(\mathbb R ^{\vec d})$
is the dual to $H^1 (\mathbb R ^{\vec d})$.
\end{theorem}

\subsection{Journ\'e's Lemma}

The explicit definition of $ \textup{BMO}$ in (\ref{e.BMOdef}) is quite difficult 
to work with. In the first place, it is not an intrinsic definition, in that 
one needs some notion of wavelet to define it.  Secondly, the supremum is 
over a very broad class of objects: All open subsets of $\mathbb R ^{\vec d}$ of finite measure. 
There are simpler definitions
(also unfortunately  not intrinsic) that in particular circumstances are 
sufficient.

Say that a collection of rectangles $ \mathcal U\subset \mathcal D_{\vec d}$
\emph{ has $ t-1$ parameters} if and only if there is a choice of coordinate $s$ so that 
for all $ R,R'\in \mathcal U$ we have $ Q _s=Q _s'$, that is the $s$th coordinate 
of the rectangles are all one fixed $ d_s$ dimensional cube.  

We then define 
\begin{equation}\label{e.BMOd-1} 
\norm f. \textup{BMO} _{-1} (\mathbb R^{\vec d}). 
\eqdef \sup _{\substack{\textup{$ \mathcal U$ has $ t-1$ } \\ \textup{parameters} }} 
\Bigl[\abs{ \operatorname {sh} (\mathcal U)} ^{-1} \sum_{R\in \mathcal U}\sum_{\vec\varepsilon\in\textup {Sig}_{\vec d}} 
\abs{ \ip f,w_R^{\vec\varepsilon},} ^2 \Bigr] ^{1/2} \,.
\end{equation}
A collection of rectangles has a \emph {shadow} given by 
$ \operatorname {sh} (\mathcal U) \eqdef \bigcup\{R\mid R\in \mathcal U\} $. 
 We use the `$ -1$' subscript to indicate that 
we have `lost one parameter' in the definition.\footnote{In the two parameter case, 
our definition of $ \textup{BMO} _{-1}$ is actually a slightly larger space than the 
more familiar rectangular $ \textup{BMO}$ space.}
Motivation for this definition 
comes from our use of induction on parameters in the proof of the lower bound 
for the commutators.  See \S\ref{ss.induct}.

 L.~Carleson 
\cite{carleson-example} produced examples of 
 functions which acted as linear functionals on
 $\operatorname H^1 (\mathbb R ^{\vec d})$ with norm one, yet 
had arbitrarily small $\textup{BMO}_{-1}$ norm. 
This example is recounted at the beginning of R.~Fefferman's article \cite{MR81c:32016}.

Journ\'e's Lemma permits us, with certain restrictions, to dominate the 
$ \textup{BMO}$ norm by the $ \textup{BMO} _{-1}$ norm.  We need a version of this 
statement with an additional refinement, see (\ref{e.d-1B}) that first appeared in 
 \cite {sarahlacey}, and is 
important to our `bootstrapping' argument in \S\ref{s.boot}.

\begin{lemma}[Journ\'e's Lemma]\label{l.journed-1}
Let  $ \mathcal U$ be a collection of rectangles  
whose shadow has finite  measure.   For any $ \eta >0$, we can 
construct $ V\supset \operatorname {sh}(\mathcal U)$ and a function 
$ \operatorname {Emb}\mid \mathcal U\longrightarrow[1,\infty )$ so that 
\begin{gather} \label{e.d-1A}
\operatorname {Emb} (R)\cdot R\subset V,\qquad R\in \mathcal U\,,
\\ \label{e.d-1B}
\abs{ V}<(1+\eta ) \abs{ \operatorname {sh} (\mathcal U)}\,,
\\ \label{e.d-1C}
\NOrm \sum_{R\in \mathcal  U} \sum_{\vec\varepsilon\in\textup {Sig}_{\vec d}}
\operatorname {Emb}(R) ^{-C }
\ip f, w_R^{\vec\varepsilon}, w_R^{\vec\varepsilon} . \textup{BMO} (\mathbb R^{\vec d}). \le
K_\eta  \norm f. \textup{BMO} _{-1}(\mathbb R^{\vec d}). .
\end{gather}
The constant $ K _{\eta }$ 
depending only on $ \eta $ and $ \vec d$, and the constant $ C$,  appearing in the 
last display, upon the  vector  $ \vec d$. 
\end{lemma}

Notice that the power on the embeddedness term in (\ref{e.d-1C})
is allowed to be quite big, a function of the parameters $ \vec d$ that 
we do not specify.
Also, concerning the conclusions, if we were to take $\operatorname {Emb}(R)\equiv1$, then 
certainly the first conclusion (\ref{e.d-1A}) would be true.  But, the last 
conclusion would be false for the Carleson examples in particular.  This choice is 
obviously not permitted in general. 

The formulations of Journ\'e's Lemma given here are not the typical ones 
found in Journ\'e's original Lemma, or  J.~Pipher's extension to the three dimensional 
case \cite {MR88a:42019}.  These papers give the more geometric formulation of these Lemmas, 
and J.~Pipher's article implicitly contains the geometric formulation 
needed to prove the Lemma above (provided one is satisfied with 
the estimate $ \abs{ V} \lesssim \abs{ \operatorname {sh} (\mathcal U)}$).  See 
Pipher \cite{MR88a:42019}. 
Lemma~\ref{l.journed-1}, as formulated above, was found in Lacey and Terwilleger \cite{math.CA/0310348}; 
the two dimensional variant (which is much easier) appeared in Lacey and Ferguson 
\cite{sarahlacey}.  The paper of Cabrelli, Lacey, Molter and Pipher 
\cite{math.CA/0412174}  surveys some  issues related to Journ\'e's Lemma. 
See in particular Sections 2 and 4.  We refer the reader to these references 
for more information on this subject.

\section{Paraproducts} \label{s.para}

The paraproducts that arise are of a somewhat general nature, and so we 
make some definitions which will permit a reasonably general 
definition of a paraproduct.

Let $ \chi (x) = (1+ \lvert  x\rvert  ^2 ) ^{-1}  $.
Let $ \chi _Q ^{(2)} =  
\operatorname {Dil} ^{(2)} _{Q} \chi $.   Say that $ \varphi $ is \emph{adapted to $ Q$} 
iff 
\begin{equation} \label{e.adapted}
\abs{ \operatorname D ^{m}\varphi (x)} \lesssim   \lvert  Q\rvert ^{-m}[ \chi ^{(2)} _{Q} (x)] ^{N}
\,, \qquad x\in \mathbb R ^{d}\,. 
\end{equation} 
This inequality should hold for all derivatives $ \operatorname D ^{m}$, 
where $ m\le d+1$, where $ d$ is the ambient dimension. 
The inequality should hold for  
all integers $ N$. The implied constant can depend upon these parameters. 
Say that $ \varphi $ has a zero iff $ \int \varphi \; dx=0$. 

We extend these definitions to functions $ \varphi $ on $ \mathbb R ^{\vec d}$. 
Say that $ \varphi $ is  \emph{adapted to $ R=\prod Q_s$} if and only if 
\begin{equation} \label{e.Tensoradapted}
\varphi (x_1, \dotsc ,x_t) = \prod _{s=1} ^{t} \varphi _s (x_s)
\,, \qquad 
\textup{where $ \varphi _s$ is adapted to $ Q_s$.}
\end{equation}
Say that $ \varphi $ \emph{has zeros in the $ s$th coordinate } if and only if 
\begin{equation} \label{e.zer0}
\int _{\mathbb R ^{d_s}} \varphi (x_1,\ldots, x_s, \dotsc x_t) \; d x_s =0
\,, \qquad  \textup{for all $ x_1, \dotsc x _{s-1},x _{s+1} \dotsc x_t$.}
\end{equation}

The main Theorem on paraproducts that we will need concerns bilinear operators 
formed in this way.  For $ j=1,2,3$ let $ \{\varphi _{j,R} \mid R\in \mathcal D_{\vec d}\} $
be three families of functions adapted to the dyadic rectangles in $\mathcal D _{\vec d}$. 
Then define 
\begin{equation*}
\operatorname B (f_1,f_2) \eqdef \sum _{R\in \mathcal D _{\vec d}} 
\frac {\ip f_1, \varphi _{1,R}, } {\abs{ R } ^{1/2} }  
\ip f_2 , \varphi _{2,R},\, \varphi _{3,R}.
\end{equation*}
The following result is due to  Journ\'e \cites{MR88d:42028,MR949001}.  Also see 
\cites{camil1,camil2,math.CA/0502334}.

\begin{theorem}\label{t.MainParaproducts}  Assume that the 
family $ \{\varphi _{1,R}\}$ has zeros in all coordinates.  For every other 
coordinate $ s$, assume that there is a choice of $ j=2,3$ for which the 
the family  $ \{\varphi _{j,R}\}$ has zeros in the $ s$th coordinate.  Then 
 the operator $ \operatorname B$ enjoys the property  
\begin{equation*}
\operatorname B\mid\textup{BMO}\times L ^{p}
\longrightarrow L^p\,, \qquad 1<p<\infty \,. 
\end{equation*}
\end{theorem}

We will refer to the function $ \varphi_1$ as the \emph{symbol} of the paraproduct. 
This function plays the same role for paraproducts as does the symbol of the 
commutator. 
Particularly relevant for us is the following reformulation of this theorem:
If $ \operatorname B_1$ and $ \operatorname B_2$ are bounded paraproducts, then 
the tensor product $ \operatorname B_1 \otimes \operatorname B_2$ is a bounded 
paraproduct for symbols  on the corresponding product $ \textup{BMO} $ space. 

In many applications of this result, the functions $ \varphi _{1,R}$, acting 
on the symbol of the paraproduct, will be product wavelets.  

\bigskip 

A more particular form of the upper bound on commutators plays a role in 
both the upper and lower bounds for our Main Theorem.  We state this variant 
of Theorem~\ref{t.MainParaproducts} for our use below.  In particular, the 
estimate (\ref{e.localized}) is used in the lower bound. 
It  holds when the symbol and the function the paraproduct is applied 
to have `separated wavelet support' in the sense of (\ref{e.Localized}).

For a subset of coordinates  $ J\subset \{1,\dotsc,t\}$ set 
\begin{align} \label{e.Uu}
\operatorname F_{\vec l,J} &\eqdef\sum_{\vec\varepsilon\in\textup{Sig}_{\vec d}}\, \sum_{\substack{\vec k\in \mathbb Z ^{t}\\
k_s=l_s\,, \ s\in J \\ k_s\ge l_s\,, s\not\in J}}
\Delta \operatorname F_{\vec k} \,,
\\  
\label{e.DeltaUu}
\Delta \operatorname F_{\vec k} 
& \eqdef 
\sum_{\vec\varepsilon\in\textup{Sig}_{\vec d}} 
\sum _{\substack{R\in \mathcal D _{\vec d}\\ \abs{ Q_s} = 2 ^{k_s} }} 
w^{\vec\varepsilon}_R \otimes w^{\vec\varepsilon}_R \,. 
\end{align}
For those coordinates $ s\in J$, we take the wavelet projection onto 
that scale, while for those coordinates $ s\not\in J$, we sum larger scales. 
That means that we lose the zero in the coordinates not in $ J$. 

Write $ R' \lesssim _{J} R $ if and only if  $ \abs{ Q' _{s}}\le \abs{ Q _{s}}$ for 
$ s\not\in J$ and $ \abs{ Q'_s}=\abs{ Q_s}$ for $ s\in J$.

\begin{theorem}\label{t.Ud} Let $ \operatorname T$ be a product 
Calder\'on--Zygmund kernel as in Theorem~\ref{t.T-Journe}. For all $ J\subset \{1,\dotsc,t\}$, 
and $ \vec k\in \mathbb Z ^{t}$ with 
\begin{equation}\label{e.vecK}
3\le k_s \le 8\,,\,  s\not\in J\,, \qquad   -8\le k_s \le 8\,,\, s\in J\,.  
\end{equation}
We have 
\begin{equation} \label{e.Ud}
\Norm \sum _{\vec l\in \mathbb Z ^{t}} (\Delta 
\operatorname F_{\vec l}\, b)\cdot  \operatorname T\operatorname F_{\vec l+\vec k,J} \varphi  .2. 
\lesssim \norm b.\textup{BMO}(\mathbb R ^{\vec d}). \norm \varphi .2. .
\end{equation}
Moreover, suppose we have the following separation condition: 
Fix an integer $ A>0$. 
Suppose that 
\begin{equation}\label{e.Localized}
\textup{if for $\vec\varepsilon$ and $\vec{\varepsilon'}$, $\ip b,w_ {R'}^{\vec{\varepsilon'}},\neq0$, $ \ip \varphi , w _{R}^{\vec\varepsilon},\neq0$  with 
$ R' \lesssim _{J} R$, then $ AR \cap R'=\emptyset $. }
\end{equation}
We then have the estimate 
\begin{equation} \label{e.localized}
\NOrm \sum _{\vec l\in \mathbb Z ^{t}} (\Delta  \operatorname 
F_{\vec l} b)\cdot \operatorname T
\operatorname F_{\vec l+\vec k,J} \varphi  .2. 
\lesssim A ^{-100t}\norm b.\textup{BMO}(\mathbb R ^{\vec d}). \norm \varphi .2. .
\end{equation}
Implied constants are independent of the choice of $ \vec k$. 
\end{theorem}

The operator in (\ref{e.Ud}), though it fits into the category of 
paraproducts, it does not fit the precise definition we have given of 
a paraproduct above, and so we will postpone the proof of this Theorem 
until the end of this section.

\subsection*{Shift Operators}

There are different  types of `shifts' on wavelets that also enter into our 
considerations. These are shifts of signature, scale and location.  We discuss 
each of these.

Define a `signature shift' operator by a map $ \epsilon \mid  \textup{Sig} _{\vec d} 
\times \mathcal D _{\vec d} \longrightarrow \textup{Sig} _{\vec d}$, 
which $ \epsilon ( \cdot , R)$ is one to one for each rectangle $ R$. Then 
the operator is defined first on wavelets by 
\begin{equation*}
\operatorname \Sh _{ \textup{Sig}, \epsilon} (w ^{\vec\varepsilon } _{R})= w ^{\epsilon (\vec\varepsilon, R)
} _{R}
\end{equation*}
and then extended linearly.   The boundedness properties of these operators are 
straightforward. 

\begin{proposition}\label{p.sigShift} We have the estimate 
\begin{equation*} 
\norm \operatorname \Sh _{ \textup{Sig}, \epsilon} .p\to p. \lesssim C_p\,, 
\qquad  1<p<\infty \,. 
\end{equation*}
\end{proposition}

The proof follows immediately from the Littlewood--Paley inequalities.  We omit the details. 

\smallskip 

Define a `scale shift' operator by a one to one map $ \sigma _{\rho }
\mid \mathcal D _{\vec d} \longrightarrow \mathcal D _{\vec d} $ that sends each dyadic rectangle $ R $ into 
a unique $ \sigma (R)\subset R $, so that the ratio 
$ \rho   = \abs{\sigma (R)}/ \abs{ R}  $ is independent of $ R$.  
The parameter of this operator is $ \rho $. 
Define a corresponding linear operator $ \operatorname \Sh _{\textup{scale}, \rho } $ by 
\begin{equation*}
\operatorname \Sh _{\textup{scale},\rho } (w ^{\vec\varepsilon } _{R}) \eqdef  
\sqrt {\rho } \cdot w ^{\vec\varepsilon } _{ \sigma (R)} 
\end{equation*}
and the operator is then uniquely defined by linearity.   Our observation is that 
this shift is a uniformly bounded operator on  product $ \textup{BMO} $.

\begin{theorem}\label{t.scaleShift}  The operators $ \operatorname \Sh _{\textup{scale},\rho}$ 
map $ \textup{BMO} (\mathbb R ^{\vec d})$ to itself.  Moreover for all $ \kappa >0$
we have the estimate 
\begin{equation*}
\norm \operatorname \Sh _{\textup{scale}, \rho }. \textup{BMO} \to 
\textup{BMO} . \lesssim  \rho  ^{-\kappa  }\,.
\end{equation*}
\end{theorem}


\begin{proof}
Given $ f\in \textup{BMO}$, and open set $ U\subset \mathbb R ^{\vec d}$, 
consider the set 
\begin{equation*}
V \eqdef \{ \operatorname M \mathbf 1 _{U} > c \rho  \}
\end{equation*}
where $ \operatorname M$ is the strong  $ t$ parameter maximal function appropriate 
to this setting, namely 
\begin{equation*}
\operatorname M f \eqdef \sup _{R\in \mathcal D _{\vec d}} 
\frac {\mathbf 1 _{2R}} {\abs{ 2R}} \int _{2R} \abs{ f (y)}\; dy\,. 
\end{equation*}
Observe that for appropriate $ c$ if $ \sigma (R)\subset U$ then we have 
$ R\subset V$.  

We can estimate 
\begin{align*}
\sum _{\vec\varepsilon \in \operatorname {Sig} _{\vec d}}
\sum _{R\subset U} \abs{\ip  \operatorname \Sh _{\textup{scale}, \rho }f,  w ^{ \vec\varepsilon }_R, } ^2 
&= 
\rho   \sum _{\vec\varepsilon \in \operatorname {Sig} _{\vec d}}
\sum _{R\subset U} \abs{ \ip f  , w ^{\vec\varepsilon } _{\sigma ^{-1} (R)} , } ^2 
\\
&\le 
\rho   \sum _{\vec\varepsilon \in \operatorname {Sig} _{\vec d}}
\sum _{R\subset V} \abs{ \ip f  , w ^{\vec\varepsilon } _{R}, } ^2
\\
&\le \rho   \norm b. \textup{BMO}. ^2 \abs{ V}\,. 
\end{align*}
It remains to estimate $ \abs{ V} $ in terms of $ \abs{ U}$.

Using the $ L^p $ mapping properties of the maximal function, we can estimate 
\begin{equation*}
   \abs{ V} \lesssim \rho  ^{-p} \abs{ U}\,.
\end{equation*}
Taking $ p=1+\kappa $ proves our theorem. 
\end{proof}

\begin{remark}\label{r.log}  When the number of parameters $ t=1$, 
the operators $ \operatorname \Sh _{\textup{scale}, \rho }$ are in fact uniformly bounded 
on $ \textup{BMO}$ as follows from the weak $L^1 $ bound for the maximal function.
For $ t>1$, there is a logarithmic estimate.
\begin{equation*}
\norm \operatorname \Sh _{\textup{scale}, \rho } . \textup{BMO} \to 
\textup{BMO} . \lesssim (1+ \log 1/ \rho  ) ^{t}.
\end{equation*}
The strong 
maximal function we are using satisfies $ \norm \operatorname M.p\to p. \lesssim 
(p-1) ^{-t}$, aside from dimensional considerations from the individual 
components of $ \vec d$.  Using this estimate, and taking $ p-1 \simeq \abs{\log  \rho } ^{-1}$, 
the estimate above follows. 
\end{remark}

\smallskip 

We define `location shift' operators.  Let $ \lambda_n \mid \mathcal D_{\vec d} 
\longrightarrow \mathcal D_{\vec d} $ be a one to one map such that for all rectangles 
$ R\in \mathcal D _{\vec d}$,  the image rectangle $ \lambda _{n} (R)$ has the same 
dimensions in each coordinate, namely 
\begin{equation*}
\abs{ Q_s } = \abs{  \lambda _{n} (Q)_s}\,, 
\qquad 1\le s\le t\,. 
\end{equation*}
Moreover, $ \lambda (R)\subset  n R$.   The shift operator is then  
defined on wavelets by 
\begin{equation}\label{e.locationScale}
\operatorname \Sh _{\textup{loc},n} w ^{\vec\varepsilon }_R = w ^{\vec\varepsilon } _{\lambda _n (R)}, 
\end{equation}
and is then extended linearly.  The parameter of this operator is said to be 
$ n$.

The estimate we need concerns the $ L ^{p}$ norms of this operator. 

\begin{proposition}\label{p.locationShift} We have the estimates below, valid
 for all integers $ n$.  
\begin{equation*}
\norm \operatorname \Sh _{\textup{loc},n} .p\to p.  
\lesssim 
n ^{\abs{ \vec d}}
\end{equation*} 
where $ \abs{ \vec d}= d_1+d_2+ \cdots + d_t$ depends only on $ \vec d$. 
\end{proposition}

\begin{proof}
Since $ \lambda_n $ is one to one, it is clear that $  \operatorname \Sh _{\textup{loc}}$ 
is bounded with norm one on $ L^2$. 
For $ p\neq 2$ we use the Littlewood--Paley 
inequalities, together with the obvious fact that 
\begin{equation*}
\mathbf 1 _{\lambda _n (R)} \lesssim   n ^{\abs{ \vec d}}\operatorname M \mathbf 1 _{R}
\end{equation*}
for all rectangles $ R$.  Then, using the Fefferman--Stein Maximal inequality, we have 
\begin{align*}
\norm  \operatorname \Sh _{\textup{loc},n} f.p. 
	& \lesssim  
	\NOrm \Bigl[\sum _{\vec\varepsilon \in \textup{Sig} _{\vec d}} 
	\sum _{R\in \mathcal D ^{\vec d}} \frac { \abs{ \ip f, w ^{\vec\varepsilon }_R ,} ^2 } 
	{\abs{ R}} \mathbf 1 _{\lambda _n (R)}  \Bigr] ^{1/2} .p. 
\\
& \lesssim n ^{\abs{ \vec d}}
\NOrm \Bigl[\sum _{\vec\varepsilon \in \textup{Sig} _{\vec d}} 
	\sum _{R\in \mathcal D ^{\vec d}} \frac { \abs{ \ip f, w ^{\vec\varepsilon }_R ,} ^2 } 
	{\abs{ R}} (\operatorname M\mathbf 1 _{R} ) ^2  \Bigr] ^{1/2} .p. 
\\
& \lesssim n ^{\abs{ \vec d}}
\NOrm \Bigl[\sum _{\vec\varepsilon \in \textup{Sig} _{\vec d}} 
	\sum _{R\in \mathcal D ^{\vec d}} \frac { \abs{ \ip f, w ^{\vec\varepsilon }_R ,} ^2 } 
	{\abs{ R}}  \mathbf 1 _{R}  \Bigr] ^{1/2} .p. 
\\
& \lesssim n ^{\abs{ \vec d}} \norm f.p. .
\end{align*}

\end{proof}

\subsection*{Generalized Paraproducts}

Experience shows that paraproducts arise in a variety of ways.  They do 
in this paper, and in this section, we adopt a notation to formalize 
the different ways that the paraproducts arise.  

Given an operator $ \operatorname P$  acting on $ L ^{p}( \mathbb R ^{\vec d})$, we set 
\begin{equation}\label{e.paraNorm}
\norm \operatorname P. \textup{Para}. 
=\inf \Bigl\{ \sum _{\epsilon } \sum _{\rho    }  
\sum _{n }    n ^{\abs{ \vec d}} \rho ^{-1/\abs{ \vec d}}
\abs{c (\epsilon ,\rho,n) } \Bigr\}
\end{equation}
where the infimum is taken over all representations 
\begin{align*}
\operatorname P f=  \sum _{\epsilon } \sum _{\rho    }  
\sum _{n }   
c (\epsilon ,\rho  ,n)    \cdot  
\operatorname B _{\epsilon , \rho   , n} ( 
\operatorname \Sh _{\textup{scale}, \rho} b, \operatorname \Sh _{\textup{Sig}, \epsilon } 
\operatorname \Sh _{\textup{loc}, n} f)\,.
\end{align*}
In this display,
the operators $ \operatorname B _{\epsilon , \rho   , n}  $ are paraproducts 
as in Theorem~\ref{t.MainParaproducts}, with norm at most one. 
The operators $ \operatorname \Sh _{\textup{scale}, \rho}$ are scale 
shift operators, with parameter $ \rho $; 
the $ \operatorname \Sh _{\textup{Sig}, \epsilon }$ are signature shift operators; 
and $ \operatorname \Sh _{\textup{loc}, n}$ are location shift operators 
of parameter $ n$.

We may combine the different results of this section into the estimate 
\begin{equation}\label{e.genearlizedParaproduct}
\norm \operatorname P. p\to p. \lesssim \norm \operatorname P. \textup{Para}. 
\,, \qquad 1<p<\infty \,. 
\end{equation}
Examples of how to use this norm are in the next proof.

\begin{proof}[Proof of Theorem~\ref{t.Ud}.] 
We will assume that the Calder\'on--Zygmund operator $ \operatorname T$ is the 
identity.  It is straightforward to supply the necessary additional 
details to accommodate the general case.

The  `father wavelet'   $W$ permits us to rewrite the operator in (\ref{e.Uu}).  
For a subset of coordinates $ J\subset \{1,\dotsc,t\}$ we set 
\begin{equation*}
W _{R,J}(x_1,\dotsc ,x_t) \eqdef \prod _{s\in J} w_{Q_s}^{\varepsilon_s}(x_s)\cdot 
\prod _{s\not\in J} W _{Q_s}(x_s)\,.
\end{equation*}
Thus, in the coordinates in $ J$ we take the Meyer wavelet, 
and for those coordinates not in $J$ we take a father wavelet. 
In particular, $ \operatorname F_{\vec l,J}$ as defined in (\ref{e.Uu}) is 
\begin{equation*}
\operatorname F_{\vec l,J} = 
\sum_{\vec\varepsilon\in\textup{Sig}_{\vec d}}\, \sum_{\substack{R \in \mathcal D ^{\vec d} 
\\ \abs{ R_s}= 2 ^{l_s}}}
W _{R,J} \otimes W _{R,J}\,. 
\end{equation*}

 Let $ \vec k\in \mathbb Z ^{\vec t}$ be as in (\ref{e.vecK}). 
Let  $ R, R'$ be dyadic rectangles  with 
\begin{equation}\label{e.Pa}
\abs{ Q_s} = 2 ^{k_s} \abs{ Q'_s}\,, \quad 1\le s \le t \,, 
\qquad 
A \simeq \operatorname M \mathbf 1 _{R} (c (R')) \,. 
\end{equation}
The function 
\begin{equation*}
\zeta _{R,R',J} \eqdef  A ^{-N }\sqrt {\abs{ R}} w_R \cdot W _{R',J}
\end{equation*}
is adapted to $ R$, in the sense of (\ref{e.Tensoradapted}).  Here $ N $ is a fixed large constant 
depending upon $ \vec d$.

The assumption (\ref{e.vecK}) plays an essential role in describing the 
zeros of the function $ \zeta _{R,R',J}$. 
$ W_{R',J}$ has zeros for $ s\in J$, but certainly does not have zeros for $ s\not\in J$. 
The properties of the Meyer wavelet, and in particular (\ref{e.Fourier}), 
along with the assumption on $ \vec k$ then imply that $ \zeta _{R,R',J}$ 
has zeros for $ s\not\in J$.

Now, consider a map $ \pi_A \mid \mathcal D_{\vec d} \longrightarrow \mathcal D_{\vec d}$ such that 
the pairs $ R, \pi (R)$ satisfy (\ref{e.Pa}). Set $ \mu (\pi )=A$ where 
$ A$ is as in (\ref{e.Pa}). The operator 
\begin{equation*}
\operatorname B _{\pi } (b,\varphi ) \eqdef \sum _{R\in \mathcal D_{\vec d}} 
\frac {\ip b, w_R,} {\sqrt {\abs{ R}}} \, \ip \varphi , W _{ \pi (R),J} ,  
\, \zeta _{R,\pi (R),J}
\end{equation*}
is a paraproduct, composed with a change of location operator.  Note that 
the function that falls on $ b$ has zeros in all coordinates; the function 
that falls on $ \varphi $ has zeros for $ s\in J$, and $ \zeta _{R, \pi (R)}$ 
has zeros for $ s\not\in J$.  It is then clear that 
\begin{equation}\label{e.mu}
\norm \operatorname B _{\pi }. \textup{Para}. \lesssim \mu (\pi) ^{N} \,.  
\end{equation}

Now, a moments thought reveals that we can write, for appropriate choices 
of $ \pi _{v}$, 
\begin{equation} \label{e.MU}
 \sum _{\vec l\in \mathbb Z ^{t}} (\Delta 
\operatorname F_{\vec l}\, b)\cdot 
\operatorname F_{\vec l+\vec k,J} \varphi  
=\sum _{v=1} ^{\infty } \operatorname B _{\pi_v } (b, \varphi )\,. 
\end{equation}
Moreover, for all $ 0<A<1$, the number of $ \pi_v $ occurring in the sum 
above with $ \mu (\pi _v) \simeq A$ is at most $ A ^{-C}$ where $ C$ depends 
upon $ \vec d$.  But then from (\ref{e.mu}), it is clear that (\ref{e.Ud}) holds.

\bigskip 

The second conclusion of the Lemma, (\ref{e.localized}), is quite important to the proof 
of our lower bounds on commutator norms.\footnote{Estimates of this type are also 
important to detailed information about norm bounds for paraproducts.  See 
\cite[\S4.3]{math.CA/0502334}.}  But with the assumption (\ref{e.Localized}),
note that we can 
again have the equality (\ref{e.MU}), but with this additional property: For all 
$ v$, we have $ \mu (\pi _v) \lesssim A$.  It is then clear that (\ref{e.localized}) 
holds. 

\end{proof}

\section{The Upper Bound } \label{s.upper}

Let $ K$ be a standard Calder\'on--Zygmund convolution 
kernel on $ \mathbb R ^{d} \times 
\mathbb R ^{d}$.  This means that the kernel is a distribution that 
satisfies the estimates below for $ x\neq y$
\begin{equation} \label{e.CZ}
\begin{split}
\abs{ \nabla ^{j} K (y)} &\le N  \abs{ y} ^{-d-j}\,, \quad j=0,1,2, \dotsc, d+1\,. 
\\
\norm \widehat K . L ^{\infty } (\mathbb R ^{\vec d}). &\le N  \,. 
\end{split}
\end{equation}
The first  estimate combines the standard size and 
smoothness estimate.\footnote{Our proof requires  a large number of  derivatives  
on the kernel, due to an argument in \S\ref{s.oneParameter}.}
The last, and critical, assumption is equivalent 
to assuming that the operator defined on Schwartz functions by 
\begin{equation*}
\operatorname T _{K} f (x) \eqdef \int K (x-y) f (y)\; dy
\end{equation*}
extends to a bounded operator on $ L^2 (\mathbb R ^{d})$.  
The least constant  $ N$ satisfying the inequalities (\ref{e.CZ}) 
and $ \norm \operatorname T _{K}.2\to 2. \le N $ 
is some times 
referred to as the \emph{Calder\'on--Zygmund norm of $ K$.}

 Now let $ K_1,\dotsc,K_t$ be a collection of Calder\'on--Zygmund kernels, 
 with $ K_s$ defined on $ \mathbb R ^{d_s} \times \mathbb R ^{d_s}$.  
 It is not obvious that the corresponding tensor product operator 
 \begin{equation*}
T _{K_1} \otimes \cdots \otimes T _{K_t}
\end{equation*}
is a bounded operator on $ L^p (\mathbb R ^{\vec d})$.  This is a consequence of the 
multiparameter Calder\'on--Zygmund theory. 
This is a basic fact for us, so we state it here.  

\begin{pCZ}
\label{t.T-Journe} 
Let $ K_1,\dotsc,K_t$ be a collection of Calder\'on--Zygmund convolution kernels, 
 with $ K_s$ defined on $ \mathbb R ^{d_s} \times \mathbb R ^{d_s}$. 
 Then 
\begin{equation*}
\operatorname T _{K_1} \otimes \cdots \otimes \operatorname T _{K_t}
\end{equation*}
extends to a bounded linear operator from $ L^p (\mathbb R ^{\vec d})$
to itself for all $ 1<p<\infty $.  
\end{pCZ}

It is also not at all clear that the multiparameter commutators are bounded 
operators, even in the case of the Riesz transforms.
Thus, this is 
one of the principal results of this paper. 

\begin{theorem}\label{t.upper}  We have the estimates below, valid for $ 1<p<\infty $. 
\begin{equation}\label{e.BMOupper}
\norm  [ \operatorname T _{K_1}, 
\cdots [ \operatorname T _{K_t},\operatorname M_b]
\cdots ] .p\to p. \lesssim  \norm b.\textup{BMO}.\,.
\end{equation}
By $ \textup{BMO}$, we mean Chang--Fefferman $ \textup{BMO}$. 
The implied constant depends upon the vector $ \vec d$, and the 
Calder\'on--Zygmund norm of the $ \operatorname T _{K_s}$. 
\end{theorem}

There are two corollaries of this that we will use below.  
For a symbol $ b$  on $ \mathbb R ^{\vec d}$, define 
\begin{equation}\label{e.RieszNorm}
\norm b. \textup{Riesz},p. \eqdef 
\sup 
\norm  [ \operatorname R_{j_1}, 
\cdots [ \operatorname R _{j_t},\operatorname M_b]
\cdots ] .p\to p.\,, 
\qquad 1<p<\infty \,. 
\end{equation}
where the supremum is formed over all choices of Riesz transforms $ \operatorname 
R _{j_s}$ for $ 1\le j_s\le d_s$.\footnote {In the case that $ d_s=1$, the Riesz 
transforms reduce to the Hilbert transform.}
The Riesz transforms of course fall 
under the purview of the Theorem above, so we see that 
$ \norm b.\textup{Riesz},p. \lesssim \norm b. \textup{BMO}.$.  
This is half of our Main Theorem, and the other half  is the reverse inequality.

\medskip 

We will have need of another class of singular integral operators 
besides the Riesz transforms, with the Fourier transform of these 
kernels--the symbol of the kernel--being well adapted to a \emph{cone} in $ \mathbb R ^{d}$.

Suppose that the dimension $ d\ge2$. 
A \emph{cone} $ C\subset \mathbb R ^{d}$ 
is specified by the data $ (\xi _C, Q)$ where $ \xi _C\in \mathbb R ^{d}$ 
is a unit vector referred to as the \emph{direction of the cone} 
and $ Q\subset \mathbb R ^{d-1}$ is a cube centered at the origin.  The 
cone consists of all vectors $ \theta $ given in orthogonal coordinates 
$ (\theta _\xi \xi, \theta _\perp)$ with $ \theta _{\xi }= \theta \cdot \xi $, 
and $ \theta _\perp \in \theta _\xi Q$.   
   For $ 0<\lambda $ by $ \lambda C$ we mean the cone with data $ (\xi _C, \lambda Q)$. 
 By the \emph{aperture of $ C$} we mean $ \abs{ Q}$.

The Fourier restriction 
operator specified by $C$ should be bounded on all $ L^p$ spaces,  Namely 
the operator defined by 
\begin{equation} \label{e.coneProjection}
\widehat { \operatorname P_C  }f \eqdef \mathbf 1 _{C} \widehat f  
\end{equation}
should admit a uniform bound on all $ L^p (\mathbb R ^{d})$ spaces.  
By taking the boundary of the cone to be a cube this is certainly the case: 
Compositions of Fourier projections onto half spaces yields 
$ \operatorname P_C $, so it 
will not be given by composition with respect to a (one parameter)
Calder\'on--Zygmund kernel as in (\ref{e.CZ}).

For a cone in $ C\subset 
\mathbb R ^{d}$, we fix a Calder\'on--Zygmund kernel $ K_C$ which satisfies the 
size and smoothness assumptions above, and in addition, 
\begin{equation} \label{e.kappa}
\mathbf 1 _{C} \le \widehat {K_C} \le \mathbf 1 _{ (1+\kappa )C }\,.  
\end{equation}
Here, we introduce a small parameter 
$ \kappa $ which will depend upon dimension $ \vec d$.
Moreover, we choose the cone operator to make a sufficiently smooth transition from 
$ 0$ to $ 1$ that the operator  $ \operatorname { T} _{C}$ with symbol
given by $ K_C$
defines a Calder\'on--Zygmund operator, bounded on all $ L^p$, $ 1<p<\infty $.

There is however an essential point to observe: That the kernel $ K_C $ 
satisfies the Calder\'on--Zygmund estimates (\ref{e.CZ}), but with constants 
that tend to infinity as the aperture of the cone tends to infinity.  In the limit, 
the kernels $ K_C$ tend to a projection of a one dimensional Calder\'on--Zygmund 
kernel.\footnote {The operators admit uniform $ L^p$ bound in the aperture, but 
we need to apply a Theorem of Song-Ying Li \cite{MR1373281} which only applies if the 
kernels are Calder\'on--Zygmund on $ \mathbb R ^{d}$. }

But, with the aperture fixed,
in each dimension, we can choose these kernels to be rotations of one 
another, so that they admit uniform bounds in $ L^p (\mathbb R ^{d})$.  
We will refer to the operator $ \operatorname T_C $ given by convolution 
with $ K_C$ as a \emph{Cone transform.}  

As a matter of convention, in the case of $ d=1$, there are two cones, 
$ \mathbb R _{\pm}$. The Cone 
transforms are the corresponding projections onto the positive and negative 
frequency axes.  These are of course linear combinations of the 
identity and the Hilbert transform, which coincide with the Riesz transforms.

We now define a third norm on a symbol $ b$  on $ \mathbb R ^{\vec d}$
\begin{equation}\label{e.ConeNorm}
\norm b. \textup{Cone},p. \eqdef 
\sup 
\norm  [ \operatorname T_{C_1}, 
\cdots [ \operatorname T _{C_t},\operatorname M_b]
\cdots ] .p\to p.\,, \qquad 1<p<\infty \,. 
\end{equation}
where the supremum is formed over all choices of Cone transforms 
$ T _{C_s}$ with $ C_s\subset \mathbb R  ^{d_s}$ in which the aperture of the 
cone is fixed.\footnote{Later in the proof, we will specify an aperture.} 
It follows that we also have $ \norm b. \textup{Cone},p. \lesssim \norm b. \textup{BMO}.$.
This is an important observation for us, so let us formalize it in the following 
Corollary of Theorem \ref{t.upper}, which includes half of our Main Theorem. 

\begin{corollary}\label{c.upper} We have the inequalities 
\begin{equation*}
\norm b. \textup{Riesz},p.\,,\, \norm b. \textup{Cone},p. 
\lesssim 
\norm b. \textup{BMO}. \,, \qquad 1<p<\infty \,. 
\end{equation*}
For the inequality concerning Cone operators, the implied constant depends 
upon the aperture of the cones.
\end{corollary}

\begin{remark}\label{r.needingCones}  
In the one dimensional case, a `cone' is just a projection 
onto the positive axis say, and most of the considerations of this 
section are not needed.  For the sake of exposition, in this section 
we will assume that all the coordinates of $ \vec d=(d_1,\dotsc,d_t)$ 
are at least two.  The case when some coordinates are one 
is technically easier, but more difficult in terms of accommodating the 
general argument into the notation. 
\end{remark}

Let us formalize the extended version of our Main Theorem. 

\begin{extended}\label{t.extended} For all $ t\ge 1$ and choices of $ \vec d$ we have 
\begin{equation*}
\norm b. \textup{Riesz},p. \simeq  \norm b. \textup{Cone},p. 
\simeq 
\norm b. \textup{BMO}. \,, \qquad 1<p<\infty \,. 
\end{equation*}
The implied constants depend upon the vector $ \vec  d$ and the aperture of the cone. 
\end{extended}

We find it necessary to prove the equivalence  between the $ \textup{BMO}$ 
and Cone norms in order to deduce the equivalence with the Riesz norm.

\subsection{A One Parameter Result}\label{s.oneParameter}

A commutator is a special form of a paraproduct.  Our approach to 
Theorem~\ref{t.upper} is obtain a decomposition of a one parameter commutator  
into a sum of paraproducts.  The tensor product of the elements of our 
decomposition are themselves bounded operators, so we can then pass to the 
multiparameter statement of the Theorem. 

\begin{remark}\label{r.tensors} 
The multiparameter setting is related to the tensor products of dilation groups. 
An essential difficulty is that the tensor product of bounded operators need not be bounded.  
See \cite{MR837350}.  
And so it will be incumbent upon us to describe sufficient conditions on the operators 
for the  tensor products to be bounded, and reduce the commutators above to these 
settings. 
\end{remark}

A result of this type, expressing a commutator as a sum of paraproducts, is 
known to experts, and has been used in \cite{MR1349230}, and may well have 
been formulated in this way before. 

\begin{proposition}\label{p.oneParameter}
For any Calder\'on--Zygmund kernel satisfying (\ref{e.CZ}),
and symbol $ b$ we can write the commutator $ [ \operatorname T _K, \operatorname M_b]$
as an absolutely convergent sum of paraproducts  composed with 
signature, scale and location shifts.  That is, using the notation in (\ref{e.paraNorm}), 
\begin{gather*}
\norm [ \operatorname T _K, \operatorname M_b] . \textup{Para}. \lesssim 1 
\end{gather*}
\end{proposition}

\begin{proof}[Proof of Theorem~\ref{t.upper}] 
The Proposition above shows that a commutator is the absolutely convergent 
sum of bounded paraproducts.  The result of Journ\'e, Theorem~\ref{t.MainParaproducts}, 
is that the tensor product of 
bounded paraproducts is bounded.  As the commutators in our Theorem 
act on a tensor product space, we see that the commutators can be written 
as an absolutely convergent sum of tensor products of bounded paraproducts.  Hence, 
the Theorem follows. 
\end{proof}

\begin{proof}
A basic fact here is that 
if $ \phi $ is adapted to a cube $ Q$, then so is $ \operatorname T_K \phi $.
 Clearly, $ \operatorname T_K \phi $ has a zero.  
This in particular shows that for a paraproduct operator $ \operatorname  B$, we have 
\begin{equation*}
\norm \operatorname T _{K} \circ \operatorname B. \textup{Para}.
+\norm  \operatorname B \circ \operatorname T _{K} . \textup{Para}.
\lesssim 
\norm \operatorname B. \textup{Para}.
\end{equation*}

As we are working with convolution operators, we could use a  classical 
Littlewood--Paley decomposition method to prove this result.  We have however 
already introduced wavelets (which are essential later in this paper) so 
we prefer that method here.  

We recall that the Meyer wavelet  $ w$ in one dimension has Fourier transform 
identically equal to zero on a neighborhood of the origin. It follows 
from the rapid decrease of the wavelet that we then have 
\begin{equation} \label{e.wPerp}
\int _{\mathbb R  } x ^{k} w (x)\; dx =0 \,, 
\qquad k>0\,. 
\end{equation}
That is, the wavelet is orthogonal to all polynomials in $ x$.  This property 
extends to the multidimensional Meyer wavelet.

Set
\begin{equation}\label{e.FatherProjection}
\operatorname F _j \eqdef  \sum _{\varepsilon \in \operatorname {Sig} _{d}} 
\sum _{\abs{ Q}\ge 2 ^{jd}} w ^{\varepsilon }_Q \otimes w^{\varepsilon }_Q
\end{equation}
be the Father wavelet projection.  And set 
\begin{equation*}
 \Delta  \operatorname F_j \eqdef 
 \operatorname F_j - \operatorname F _{j+1} 
 =\sum _{\varepsilon \in \operatorname {Sig} _{d}}
 \sum _{\abs{ Q}= 2 ^{jd}} w ^{\varepsilon }_Q \otimes w^{\varepsilon }_Q
\end{equation*}
be the projection onto the wavelets of scale $ 2 ^{j}$.  

The property (\ref{e.Fourier}) is relevant to us.  In particular, it follows 
from this that we have the Fourier transform of the product  
$ \Delta _j \operatorname  F_j b \cdot  \operatorname F_{j+3} f$ is 
localized to   $ 2 ^{-j-3}\le \abs{ \xi }\le 2 ^{-j+3}$. 

We expand the commutator in these wavelet projections.  Thus, 
\begin{align*}
[ \operatorname T _{K} , \operatorname M_b ] f
=\sum _{j,j'} 
[ \operatorname T _{K} , \operatorname M_{ \operatorname {\Delta F} _j b} ]  
 \operatorname{\Delta F} _{j'} f.
\end{align*}
The principal term arises from $ j+3< j' $, where we do not have 
any cancellation in the commutator, and we write 
\begin{align*}
\sum _{j+3<j'} 
[ \operatorname T _{K} , \operatorname M_{ \operatorname {\Delta F} _j b} ]  
 \operatorname{\Delta F} _{j'} f
&= \operatorname T _{K} \circ  \operatorname  B_1 (b,f)- \operatorname 
B_2 (b,  f)\,,
\\ 
\operatorname  B_1 (b,\phi) & \eqdef 
\sum _{j} \operatorname {\Delta F_j} b \cdot \operatorname F _{j+3} \phi\,,
\\
\operatorname B_2(b,\phi) & \eqdef 
\sum _{j} \operatorname {\Delta F_j} b \cdot  \operatorname T_K \circ \operatorname F _{j+3} \phi\,,
\end{align*}
It is important that the 
product $\operatorname {\Delta F_j} b \cdot \operatorname F _{j+3} f $ 
have no Fourier support in a neighborhood of the origin that has diameter 
proportional to $ 2 ^{-j} $. 
Certainly, $ \operatorname B_1$ is a paraproduct. 
It follows that $ \operatorname T _{K}\circ \operatorname B_1$ is as well. 
Upon inspection, 
one sees that $ \operatorname B_2$ is a paraproduct. 
It is also straightforward to verify that 
$
\norm\operatorname T _{K}\circ  \operatorname B_1 . \textup{Para}.+\norm \operatorname B_2 . \textup{Para}.
\lesssim 1
$.

In the remaining cases we expect terms which are substantially smaller. 
The principal point is this estimate.  For 
$ \varepsilon,\varepsilon '\in 
\operatorname {Sig} _{d}$, 
\begin{equation}\label{e.QQ'}
\abs{ [  \operatorname T_K, \operatorname M _{w _{Q} ^{\varepsilon } } ] 
w _{Q'} ^{\varepsilon '} (x)} 
\lesssim  
\bigl[\tfrac { \abs{ Q}} {\abs{ Q'}} \bigr] ^{1+1/2d} 
\bigl( 1+  \tfrac{ \operatorname {dist} (Q,Q')}{ \abs{ Q} ^{1/d} } 
\bigr) ^{-N} \abs{ Q} ^{-1/2}
[\chi _{Q'} ^{(2)} (x)] ^{N}\,,
\qquad 
2 ^{3d }\abs{ Q}\ge \abs{ Q'} \,. 
\end{equation}
Here, $ \chi _{Q'} ^{(2)}$ is as in (\ref{e.adapted}). 
In the language of the section on paraproducts, this shows that a large 
constant times this function 
is adapted to the cube $ Q'$.    The power $ -N$ on the term involving distance 
holds for all large $ N$; a power of $ N>d$ is required;  The power $ 1+1/2d$ 
on the ratio  $ \abs{ Q}/\abs{ Q'}$ follows from the number of derivatives 
we have on the kernel in (\ref{e.CZ}); some power larger than one is required. 

With the inequality (\ref{e.QQ'}), it is easy to verify that 
\begin{equation*}
\norm [ \operatorname M_b, \operatorname T_K] \cdot -
  \operatorname T_K \operatorname  B_1 (b, \cdot  )+ \operatorname 
B_2 (b,\cdot ) . \textup{Para}. 
\lesssim 1\,. 
\end{equation*}

\bigskip 

The proof of (\ref{e.QQ'}) is taken in two steps.  We have 
\begin{equation*}
\abs{ [  \operatorname T_K, \operatorname M _{w _{Q} ^{\varepsilon } } ] 
w _{Q'} ^{\varepsilon '} (x)} 
\lesssim 
\bigl( 1+  \tfrac{ \operatorname {dist} (Q,Q')}{ \abs{ Q} ^{1/d} } 
\bigr) ^{-N} \abs{ Q} ^{-1/2}
[\chi _{Q'} ^{(2)} (x)] ^{N}\,,
\qquad 
2 ^{3d }\abs{ Q}\ge \abs{ Q'} \,. 
\end{equation*}
That is, we do not have the term involving  $ \abs{ Q}/\abs{ Q'}$ appearing on the right. 
This estimate is easy to obtain, and we omit the details.  

The second estimate is 
\begin{equation} \label{e.qq'}
\abs{ [  \operatorname T_K, \operatorname M _{w _{Q} ^{\varepsilon } } ] 
w _{Q'} ^{\varepsilon '} (x)} 
\lesssim 
\bigl[\tfrac { \abs{ Q}} {\abs{ Q'}} \bigr] ^{1+1/d}  \abs{ Q} ^{-1/2}
[\chi _{Q'} ^{(2)} (x)] ^{N}\,,
\qquad 
2 ^{3d }\abs{ Q}\ge \abs{ Q'} \,. 
\end{equation}
That is, we have a slightly larger power of  $ \abs{ Q}/\abs{ Q'}$ than is 
claimed in (\ref{e.QQ'}).  Taking a geometric mean of these two estimates 
will prove (\ref{e.QQ'}).

To see (\ref{e.qq'}), let us assume that $ \abs{ Q}=1$, which we can do as 
a dilation of $ K$ has the same Calder\'on--Zygmund norm as $ K$. Observe that the commutator above is 
\begin{equation*}
 \int \{ w _Q ^{\varepsilon } (x) - w _Q ^{\varepsilon } (y)\} K(x-y) \cdot 
 w ^{\varepsilon '} _{Q'} (y)  \; dy\,.
\end{equation*}
Write the leading term in the integral as 
\begin{equation*}
\{ w _Q ^{\varepsilon } (x) - w _Q ^{\varepsilon } (y)\} K(x-y) 
= T (x-y)+E (x,y)\,, 
\end{equation*}
where $ T (x-y)$ is the $ d$th degree Taylor polynomial of the left hand side, and 
the error term $ E (x,y)$ satisfies $ \abs{ E (x,y)} \lesssim \abs{ x-y} ^{d+1}$. 
That we have this estimate follows from our assumptions (\ref{e.CZ}) on the 
kernel $ K$.  
The wavelet $ w ^{\varepsilon '} _{Q'}$, by choice of wavelet, is orthogonal 
to the Taylor polynomial, see (\ref{e.wPerp}).  Thus, as claimed, 
\begin{align*}
\ABs{  \int \{ w _Q ^{\varepsilon } (x) - w _Q ^{\varepsilon } (y)\} K(x-y) \cdot 
 w ^{\varepsilon '} _{Q'} (y)  \; dy } 
 &\le 
 \ABs{ \int E (x,y)  w ^{\varepsilon '} _{Q'} (y)  \; dy }
 \\
 & \lesssim   \int \abs{ x-y} ^{d+1} \abs{  w ^{\varepsilon '} _{Q'} (y)} \; dy 
\\
& \lesssim 
\bigl[\tfrac { \abs{ Q}} {\abs{ Q'}} \bigr] ^{1+1/d} 
\chi _{Q'} ^{(2)} (x)\,.
\end{align*}
\end{proof}

\subsection*{An Estimate for Riesz Transforms}

For our use at the end of the proof of the lower bound on operator norms on 
Riesz commutators, we will need a more quantitative estimate on upper bounds 
of such commutators.  This estimate is most convenient to state here. 

\begin{proposition}\label{l.quantReisz}  
For all integers $ a\ge 1$, 
consider the operator 
\begin{equation*}
\operatorname U_a (f,g) \eqdef 
\sum _{\abs{ Q'}=2 ^{a}\abs{ Q}} \sum _{\varepsilon',\varepsilon  \in \operatorname {Sig} _{d}}
\ip f, w_{Q'} ^{\varepsilon' }, \ip g, w_{Q} ^{\varepsilon },  \, [ \operatorname M _{w_{Q'}
^{\varepsilon' }}, \operatorname R _{s}] 
w _{Q} ^{\varepsilon }
\end{equation*}
where $ \operatorname R_s$ is the $ s$th Riesz transform on $ \mathbb R ^{d}$. 
We have the estimate 
\begin{equation}\label{e.quantReisz} 
\norm \operatorname U_a . \textup{Para} . 
\lesssim   2 ^{-Ma}\,.
\end{equation}
This estimate holds for all $a, M>1$, and all Riesz transforms, with implied constant 
only being a function of $ M$, and the dimension $ d$. 
\end{proposition}

The proof is a simple variant on the previous proof.  
Clearly the role of the signatures is unimportant, and we will ignore the 
sum on the signatures in the argument below.  Note that $ \operatorname U_a$ is 
a paraproduct, with zeros falling on $ f$, and  zeros falling on $ g$. 
Now, $ Q$ and $ Q'$ have different scales, which means that $ w_Q$ 
and $ w _{Q'}$ are not adapted to cubes of the same scale.  

This was exactly the problem addressed with the inequality (\ref{e.QQ'}) above. 
However, the Riesz transform has  an infinitely smooth kernel.  Therefore, 
a stronger form of (\ref{e.QQ'}) holds.  Namely, for all $ m>1$, we have 
\begin{equation*}
\abs{ [  [ \operatorname M _{w_{Q'}
^{\epsilon }}, \operatorname R _{s}]  w _{Q} ^{\varepsilon } (x)} 
\lesssim  
\bigl[\tfrac { \abs{ Q}} {\abs{ Q'}} \bigr] ^{m} 
\bigl( 1+  \tfrac{ \operatorname {dist} (Q,Q')}{ \abs{ Q'} ^{1/d} } 
\bigr) ^{-N} \abs{ Q'} ^{-1/2}
[\chi _{Q} ^{(2)} (x)] ^{N}\,,
\qquad 
\abs{ Q'}= 2 ^{a}\abs{ Q} \,. 
\end{equation*}
The implied constant depends only upon $ m$, through the growth of the constants in the 
relevant estimates of the Riesz transform kernels.  

The proof of the estimate proceeds just as the proof of (\ref{e.QQ'}), so we 
omit the details.  The derivation of the proposition from this last estimate is routine.

\section{The Lower Bound} \label{s.lower}

We turn to the converse to Corollary~\ref{c.upper}, namely the Theorem 
below, which includes half of our Main Theorem. 

\begin{theorem}\label{t.lower} 
We have the inequalities below, valid 
for all choices of $ \vec d$. 
\begin{equation*}
\norm b.\textup{BMO}. \lesssim  
\norm b. \textup{Riesz},p.\,,\, \norm b. \textup{Cone},p. 
\,, \qquad 1<p<\infty \,. 
\end{equation*}
where the two norms are defined in (\ref{e.RieszNorm}) and (\ref{e.ConeNorm}). 
For the inequality involving the cone norm, the implied constant depends upon 
the aperture of the cone. 
\end{theorem}

\begin{remark}\label{r.p=2}  
It is enough to prove this inequality with the $ L^2$ operator norm on the right 
hand side. If a commutator is bounded from $ L^p$ to itself, then it is also 
bounded on the conjugate space $ L^{p'}$, and so by interpolation bounded on 
$ L^2$.  That is, we have the inequality $ \norm b.\textup{Riesz},2.
\lesssim  \norm b.\textup{Riesz},p.$, valid for all $ 1<p<\infty $. 
The same inequality holds for the Cone norm. 
\end{remark}

For the rest of this paper, we will denote $\norm b.\textup{Riesz},2.$ by $\norm b.\textup{Riesz}.$ and similarly for the Cone norms.

We use induction on parameters, 
namely the number of coordinates in $ \vec d$.  The base case is $ t=1$. 
Coifman, Rochberg and Weiss \cite{MR54:843} proved that 
$ \norm b. \textup{BMO}. \lesssim \norm b. \textup{Riesz}. $. 
This is a   well known result, with a concise proof.  
We find it necessary to prove the same inequality for the cones as an 
aid to proving the result about Riesz transforms.   Indeed, 
it was this part of the proof that motivated the definition of the cone norm. 

In the case $ t=1$, we indeed have the inequality $ \norm 
b. \textup{BMO}. \lesssim \norm b. \textup{Cone}. $.  This is a consequence of 
a deep line of investigation begun by Uchiyama \cite {MR0467384}, in which 
both directions of the Coifman, Rochberg, and Weiss result were extended to 
more general Calder\'on--Zygmund operators.   In particular, a 
result of Song-Ying Li gives us as a Corollary to his Theorem, 
this essential result, which completes our 
discussion of the base case $ t=1$ in our induction on parameters.

\begin{theorem}[Li \cite {MR1373281}] \label{t.li} In the case of $ t=1$, 
for all $ d\ge1$  and symbols $ b$ on $ \mathbb R ^{d}$ we have 
\begin{equation*}
\norm b. \textup{BMO}. \lesssim \norm b. \textup{Cone}. \,. 
\end{equation*}
\end{theorem}

In the inductive stage of the proof, we use the induction hypothesis to 
derive a lower bound on the commutator norms in terms of 
the $ \textup{BMO} _{-1} $ norm.  In so doing, it is very useful to use 
the equivalent Weak Factorization Theorem. 

We then `bootstrap' from this weaker inequality to the full inequality.  Namely, 
we can work with a symbol $ b$ with $ \textup{BMO}$ norm one, but with 
$ \textup{BMO} _{-1}$ norm small.\footnote{That is, the function 
$ b$ is of the type found in Carleson's examples \cite{carleson-example}.} 
With $ b$ fixed, we select an appropriate commutator which will admit 
a lower bound on its operator norm.  We select a test function which will 
show that the commutator has a large operator norm.  Verification of this 
fact will depend critically on the assumption that the symbol has 
small $ \textup{BMO} _{-1}$ norm, and the Journ\'e Lemma.

\subsection{The Initial  $ \textup{BMO} _{-1}$ Lower Bound}\label{ss.induct}

We assume that $ t\ge2$ and use the induction hypothesis to 
establish a lower bound on the Riesz and Cone norms of a symbol.  
This norm is in terms of our $ \textup{BMO}$ norm with $ t-1$ parameters.

\begin{lemma}\label{l.bmoredux} For $ t\ge 2$, assume Theorem~\ref{t.lower} in the case of $ t-1$ parameters.  Then we have the estimate 
\begin{equation} \label{e.t-1}
\norm b .\textup{BMO}_{-1}. \lesssim 
\norm b. \textup{Cone}. \,,\, \norm b. \textup{Riesz}.\,.
\end{equation}
where it is to be emphasized that the $ \textup{BMO}_{-1}$ norm on the left is the $ \textup{BMO}$ norm 
of $ t-1$ parameters. 
\end{lemma}

\begin{proof}
We only give the proof of $ \norm b .\textup{BMO}_{-1}. \lesssim 
\norm b. \textup{Riesz}. $ explicitly.  This proof  uses an 
equivalent form of the induction hypothesis, namely a 
weak factorization result on $ H^1$ in $ t-1$ parameters.  The same weak 
factorization result holds for Cone transforms.  See Li \cite {MR1373281} 
for the one parameter formulation of this result.

Using the notation of  (\ref{e.LotimesL}), 
it is a straightforward exercise in duality to demonstrate that 
\begin{equation}
\sup_{\vec{j}}\norm \operatorname C_{\vec{j}}(b,\cdot). 2\to 2. \approx \norm b . (L^2(\mathbb R^{\vec d})
\widehat\odot L^2(\mathbb R^{\vec d})) ^*. .
\end{equation}
Therefore to show (\ref{e.t-1}), it is sufficient to demonstrate that the following inequality holds,
\begin{equation}
\label{e.t-1red}
\norm b.(L^2\widehat\odot L^2)^*.\gtrsim\norm b.\textup{BMO}_{-1}.,
\end{equation}
and this will be established by relying upon the truth of the Theorem in $t-1$ parameters.

Given a smooth symbol $b(x_1,\ldots x_t)=b(x_1,x')$ of $t$ parameters, 
we assume that $\norm b.\textup{BMO}_{-1}.=1$.  As the symbol is smooth, the 
supremum in the norm is achieved by a 
collection of rectangles $\mathcal{U}$ of $\mathcal{D}_{\vec d}$ of $t-1$ 
parameters.  We can assume that the rectangles in $ \mathcal U$ agree in the 
first coordinate, to a cube $ Q\subset \mathbb R ^{d_1}$. 
As there are free dilations in each coordinate, we can 
 assume that $\abs{Q}=1$ and $\abs{\textnormal{sh}(\mathcal{U})}\approx 1$. 
Then define 
$$
\psi = \sum_{R\in\mathcal{U}}\sum_{\vec\varepsilon\in\textup{Sig}_{\vec d}}
\ip b,w_R^{\vec\varepsilon}, w_R^{\vec\varepsilon}.
$$
One notes that $\ip b,\psi,=1$.  To prove the claim, it is then enough to
demonstrate that $\norm \psi .L^2(\mathbb R^{\vec d})\widehat\odot L^2(\mathbb R^{\vec d}).\lesssim 1$.  Next observe that $\psi(x)=\psi_1(x_1)\psi'(x')$ and $\psi_1\in H^1(\mathbb R^{d_1})$ with 
$$
\norm \psi_1.H^1(\mathbb R^{d_1}). =1.
$$
For the function $\psi_1$, we use the one parameter weak factorization of 
$H^1(\mathbb R^{d_1})$ of Coifman, Rochberg and Weiss \cite{MR54:843}: 
There exists functions $f_n^j,g_n^j\in L^2(\mathbb R^{d_1})$, $n\in\mathbb N$, 
$1\leq j_1\leq d_1$, such that
$$
\psi_1=\sum_{n=1}^\infty\sum_{j_1=1}^{d_1}\Pi_{1,j_1}(f_n^{j_1},g_n^{j_1})
$$
where $\Pi_{1,j_1}(p,q):=R_{1,\, j_1}(p) q+pR_{1,\, j_1}(q)$. 
One next sees that $\psi'\in H^1(\otimes_{l=2}^{t}\mathbb R^{d_l})$ 
with norm controlled by a constant.  This follows from the 
choice of $\mathcal{U}$ and the square function characterization
of the space $H^1(\otimes_{l=1}^{t-1}\mathbb R^{d_l})$. 
By the induction hypothesis in $ t-1$ parameters, in particular that $H^1(\otimes_{l=2}^{t}
\mathbb R^{d_l})=L^2(\otimes_{s=2}^t\mathbb R^{d_s})\widehat\odot
L^2(\otimes_{s=2}^t\mathbb R^{d_s})$, we have $f_m^{\vec{j}},g_m^{\vec{j}}\in 
L^2(\otimes_{s=2}^{t}\mathbb R^{n_s})$ with $m\in\mathbb N$ and $\vec{j}$ a 
vector with $1\leq j_s\leq d_s$ for $s=2,\ldots, t$ such that
$$
\psi'=\sum_{m=1}^{\infty}\sum_{\vec{j}}\Pi_{\vec{j}}  
(f_m^{\vec{j}},g_m^{\vec{j}}),\qquad \sum_{m=1}^{\infty}\sum_{\vec{j}}
\norm f_m^{\vec{j}} .2. \norm g_m^{\vec{j}}.2. \lesssim 1.
$$

This immediately implies (\ref{e.t-1}) since $\psi=\psi_1\psi'$, and we have a weak factorization of $\psi$ with $\norm \psi .L^2(\mathbb R^{\vec d})\widehat\odot L^2(\mathbb R^{\vec d}).\lesssim 1$. 
\end{proof}

\section{The Bootstrapping Argument}\label{s.boot}

In this section, we assume that $ \norm b . \textup{BMO} _{-1}. < \delta _{-1}$, 
is very small, for a constant $ 0<\delta _{-1}<1$ to be chosen.  Under this additional 
assumption, we conclude the proof of Theorem~\ref{t.lower}.  

This proof is intricate, and indeed at this stage we find it essential to 
first prove the result for cones, namely we first prove $ \norm b. \textup{BMO}. \lesssim \norm 
b . \textup{Cone}. $.  Elements of this  proof in this case are essential to 
address the Riesz norm case.

\subsection{The Lower Bound on the Cone Norm}

This case follows the lines of the argument of Lacey and Terwilleger 
\cite{MR2176015}. (The current argument is however somewhat 
simpler.) 
We make a remark about the cone norms with different apertures.  
Given a cone $ C$ with data $ (\xi , Q)$ where $\xi  $ is a unit vector 
in $ \mathbb R ^{d}$ and $ Q$ is a cube, consider a second cone $ C'$ with  
data $ (\xi', C')$.  We can map one cone into the other with an orthogonal 
rotation and a dilation in $ d-1$ variables.  Thus, the corresponding Calder\'on--Zygmund kernels  
$ K_C$ and $ K _{C'}$ can be mapped one into the other by way of these same transformations. 

A rotation preserves the Calder\'on--Zygmund norm of the kernel, but the dilation 
does not 
since it is not uniform in all coordinates.  Nevertheless, this 
observation shows that the Cone norms associated to distinct apertures are 
comparable.  Thus, to prove our result, it suffices to demonstrate the 
existence of \emph{some} aperture for which the Theorem is true.  This we will do 
by taking a somewhat large aperture, that approximates a half space.  

For a choice of symbol $ b$ with $ \norm b. \textup{BMO}.=1$ 
and $ \norm b.\textup{BMO} _{-1}.<\delta _{-1}$, there is an associated open set $U$ for which we achieve the supremum in the
$\textup{BMO}$ norm.   
After an appropriate dilation, we can assume
$\frac{1}{2}<\abs{\textup{sh}(\mathcal U)}\leq 1$. 
Let $ \mathcal U = \{R\in \mathcal D _{\vec d} \mid R\subset U\}$. 

For a collection of rectangles $\mathcal T$ define the wavelet projection onto $ \mathcal T$ 
as 
\begin{equation*}
\operatorname P_{\mathcal T}b\eqdef\sum_{R\in\mathcal T}
\sum_{\vec{\varepsilon}\in\textup{Sig}_{\vec d}}\ip b,w_R^{\vec\varepsilon} ,w_R^{\vec\varepsilon}.
\end{equation*}
Define $ \beta = \operatorname P _{\mathcal U}b$. 
We use this function to build a  test function to demonstrate a lower bound on the Cone norm.

The purpose of the next steps is to select the cones we will use.  
This issue involves some subtleties motivated by subsequent steps in the proof. 
Of particular importance is that the selection of the cones be dependent only on 
dimension, as well as satisfy some particular estimates.  It turns out to be 
a useful device to select two distinct cones: One will be used for the 
selection of the test function for the commutator, and the other for the cones we 
use to define the commutator.

Given a cone $ C$ with data $ (\xi ,Q)$, let $ \operatorname H _{C} $ be the 
convolution operator with symbol $ \mathbf 1 _{(0,\infty )} (\xi \cdot \theta )$, 
which is to say that $ \operatorname H_C$ is  the Fourier projection onto 
a half space associated with $ C$.

\begin{lemma}\label{l.coneSelection} 
Given $ \vec d$ and $ \kappa >0$ we can select cones 
\begin{equation*}
D_s\subset C_s \subset \mathbb R ^{d_s} \,, \qquad  1\le s \le t\,. 
\end{equation*}
These cones have data $ (\xi _s, Q_s)$ and $ (\xi _s, Q'_s)$ respectively. 
They are, up to a rotation, only a function of $ \vec d$ and $ \kappa >0$, 
and they satisfy these properties.  Defining 
\begin{equation}\label{e.gammaDefined}
\gamma \eqdef \operatorname T _{D_1} \cdots \operatorname T _{D_t} \beta 
\end{equation}
we have 
\begin{align} \label{e.GAMMA}
\norm \gamma .2. &\ge 4 ^{-t}\,. 
\\  \label{e.1stDiff}
\norm \operatorname H _{D_1}\cdots \operatorname H _{D_t} \beta 
- \gamma  .4. 
 &\le  \kappa  .
\\ \label{e.2ndDiff}
\norm (\operatorname H _{C_1}\cdots \operatorname H _{C_t} 
- \operatorname P _{C_1}\cdots \operatorname P _{C_t})\abs{ \gamma } ^2  .2. 
 &\le \kappa  .
\end{align}
Notice that (\ref{e.1stDiff}) estimates an $ L ^{4}$ norm; 
and that (\ref{e.2ndDiff}) concerns the function $ \abs{ \gamma } ^2 $, and we 
are estimating the difference between the projections onto the half spaces defined 
by the cones, and the projection onto the \emph{larger} cones. 
\end{lemma}

\begin{proof} We begin with the selection of the cones $ D_s$, which 
is a   randomized procedure.
Fix a small constant $ 0<\eta <\tfrac 1 {10}$. 
In each dimension $ \mathbb R ^{d_s}$, fix an aperture 
$ Q_s$ so that the cone $ D_s$ with this aperture satisfies 
\begin{equation*}
\mathbb P ( D_s \cap S ^{d_s-1} \;|\; S ^{d_s-1})\ge \tfrac 12 -\eta \,. 
\end{equation*}
Here, $ S ^{d_s-1}$ denotes the sphere in $ \mathbb R ^{d_s}$ endowed with the 
canonical normalized surface measure.  The notation above is the  standard 
way to denote conditional 
probability.

Now, let $ D_s'$ denote a random rotation of the cone $ D_s$.  Taking 
expectations of $ L ^{2}$ norms below, we have access to the Plancherel 
identity to see that 
\begin{align*}
\mathbb E 
\norm \operatorname P _{D_1'} \cdots \operatorname P _{D_t'}\beta .2. ^2 
=
c _{\vec d} \,\mathbb E \int _{D_1' \otimes \cdots \otimes D _{t} '} 
\abs{  \widehat \beta (\vec \xi )  } ^2 \; d \vec \xi 
\ge (\tfrac12 - \eta ) ^{t}\,. 
\end{align*}
But also, we must have 
\begin{equation*}
\mathbb E 
\norm (\operatorname H _{D_1'} \cdots \operatorname H _{D_t'}-
\operatorname P _{D_1'} \cdots \operatorname P _{D_t'})\beta .2. ^2 
\le \eta ^{t}\,. 
\end{equation*}
View these statements about the $ L ^{2}$ norm of non-negative random variables. 
As concerns the first inequality, note that 
\begin{align*}
\sup _{D'_1 ,\dotsc, D'_t} 
\norm \operatorname P _{D_1'} \cdots \operatorname P _{D_t'}\beta .2. ^2 
&\le 1\,,
\end{align*}
Hence, we see that 
\begin{align*}
\mathbb P ( 
\norm \operatorname P _{D_1'} \cdots \operatorname P _{D_t'}\beta .2. ^2 \ge 
(\tfrac 14 ) ^{t}
) &\ge (\tfrac12 - \eta ) ^{t}\,, 
\\
\mathbb P (\norm (\operatorname H _{D_1'} \cdots \operatorname H _{D_t'}-
\operatorname P _{D_1'} \cdots \operatorname P _{D_t'})\beta .2. ^2  
\ge \eta ^{t/2}) &\le \eta ^{t/2}\,. 
\end{align*}

Therefore, for $ \eta $ sufficiently small, 
we can select cones $D'_1 ,\dotsc, D'_t $ so that 
\begin{align*}
\norm \operatorname P _{D_1'} \cdots \operatorname P _{D_t'}\beta .2. ^2 
&\ge (\tfrac14) ^{t} 
\\
\norm (\operatorname H _{D_1'} \cdots \operatorname H _{D_t'}-
\operatorname P _{D_1'} \cdots \operatorname P _{D_t'})\beta .2. ^2 
&\le \eta ^{t/2}\,. 
\end{align*}
On the other hand, we automatically have 
\begin{equation*}
\norm (\operatorname H _{D_1'} \cdots \operatorname H _{D_t'}-
\operatorname T _{D_1'} \cdots \operatorname T _{D_t'})\beta .8. 
\le C \norm \beta .8.
\le K
\end{equation*}
for an absolute constant $ K$.  
Keeping in mind the fact that the symbol of a cone operator $ \operatorname T _{C}$ is identically one 
on the cone $ C$, we see that we have proved (\ref{e.GAMMA}) and (\ref{e.1stDiff}). 

\medskip 

We  use proof by contradiction to find the cones $ C_s$. 
Fix the cones $ D_s$ as above, and let us suppose that (\ref{e.2ndDiff}) fails for some $ \kappa >0$.  
Then, we can find a sequence of   cones 
\begin{equation*}
C _{s} ^{k} \subsetneq C _{s} ^{k+1}\,, \qquad 1\le s \le t\,, \ k\ge 1
\end{equation*}
with data $ (\xi _k, Q' _{s,k})$, where the apertures $ Q _{s,k} '$
 increase to all of $ \mathbb R ^{d_s-1}$.  

We can also find   functions $ \beta _{k}$  satisfying 
\begin{gather*}
\norm \beta _k .2. =1\,, \qquad \norm \beta _k.8. \le K_8\,,
\\
\norm (\operatorname H _{C_1 ^{k}}\cdots \operatorname H _{C_t ^{k}} 
- \operatorname P _{C_1 ^{k}}\cdots \operatorname P _{C_t ^{k}})\abs{ \gamma_k } ^2  .2. \ge \kappa \,. 
\end{gather*}
where $ \gamma _k$ is defined as in (\ref{e.gammaDefined}).  
The constant $ K_8$ depends only on $ \vec d$, and the John Nirenberg inequality for 
$ \textup{BMO}$.

In particular, as we can assume an upper bound on the $ L ^{8} $ 
norm of the $ \beta _k$, the sequence of functions $ \{\beta _k\}$ are precompact in the 
$ L ^{2}$ topology.  Letting $ \beta  _{\infty }$ be a limit point of the sequence of functions, 
and defining $ \gamma _{\infty  }$ as in (\ref{e.gammaDefined}), 
we see that  for all large $ k$, 
\begin{equation*}
\norm (\operatorname H _{C_1 ^{k}}\cdots \operatorname H _{C_t ^{k}} 
- \operatorname P _{C_1 ^{k}}\cdots \operatorname P _{C_t ^{k}})\abs{ \gamma_\infty  } ^2  .2. \ge \kappa \,. 
\end{equation*}
But this is an absurdity, as in the limit, the symbol of this difference is supported on 
a subspace of codimension $ t$.  Therefore, (\ref{e.2ndDiff}) holds.  

\end{proof}

The cones we form the commutator of are the $ C_s$ of the previous Lemma. We will test 
the commutator against the function $ \overline \gamma $, where $ \gamma $ is  as in
(\ref{e.gammaDefined}).

By Journ\'e's Lemma, in particular Lemma~\ref{l.journed-1},
there will exist an open set $V$ which satisfies the conditions of that Lemma.  Set 
\begin{equation}\label{e.V}
\mathcal{V}\eqdef\{R:R\subset V,\ R\not\subset\textup{sh}(\mathcal{U})\}.
\end{equation} 
Finally, let $\mathcal W\eqdef \mathcal{D}_{\vec d}-\mathcal{U}-\mathcal{V}$.  
  
The function  $ \gamma $   enjoys these properties of Lemma~\ref{l.coneSelection}, 
as well as the ones below, which will 
conclude the proof. 
\begin{align}
\label{e.Obvious} 
\norm [ \operatorname T _{C_1}, \cdots [ \operatorname T_{C_t}, \operatorname M _{ P_{\mathcal U}b}
] \cdots ]\overline\gamma.2. & \gtrsim 1,
\\
\label{e.Journe} 
\norm [ \operatorname T _{C_1}, \cdots [  \operatorname T_{C_t},\operatorname M _{ P_{\mathcal V}b}
] \cdots ]\overline\gamma .2. &\lesssim \delta_J^{1/4},
\\
\label{e.Paraprod} 
\norm [ \operatorname T _{C_1}, \cdots [  \operatorname T_{C_t},\operatorname M _{ P_{\mathcal W}b}
] \cdots ]\overline\gamma.2.&\leq K_J\delta_{-1}.
\end{align}
Here $0< \delta _J<1$ is the constant associated with Journ\'e's Lemma (called $\eta$ in Lemma \ref{l.journed-1}), 
that is to be specified.  $ K_J$ is a function of $ \delta _J$. 
These estimates will lead to an absolute lower bound and prove Theorem~\ref{t.lower}.
Namely, the implied constants in each of the inequalities depend only on 
$ \vec d$, while $ \delta _{J}$ and $ \delta _{-1}$ are free to choose.  Certainly we 
can choose $ \delta _{J}$ first, and then with $ K_J$ specified in (\ref{e.Paraprod}), 
select $ \delta _{-1}$ to prove our Theorem.  

Estimate (\ref{e.Journe}) is  straightforward.  It is easy to see that
$$
\norm [ \operatorname T _{C_1}, \cdots [T_{C_t},\operatorname M _{ P_{\mathcal V}b}
] \cdots ]\overline\gamma .2.
\lesssim\norm\operatorname P_{\mathcal V}b.4.\norm \gamma  .4.
\lesssim \norm\operatorname P_{\mathcal V}b.4.\,,
$$
where the implied constant depends upon the the  $ L^4$ 
norms of the Cone transforms. 
But, by Journ\'e's Lemma \ref{l.journed-1} and construction, we have that 
$$
\norm\operatorname P_{\mathcal V}b.2.\leq\delta_J^{1/2}\, ,
\qquad\norm\operatorname P_{\mathcal V}b.\textup{BMO}.\leq 1\,,
$$
which implies that
$$
\norm\operatorname P_{\mathcal V}b.4.\leq\delta_J^{1/4}.
$$
These together give (\ref{e.Journe}).

We turn to the verification of (\ref{e.Obvious}) and provide a lower 
bound for the $ L^2$ norm below
\begin{equation*}
\norm [\operatorname T_{C_1}, \cdots  [\operatorname T_{C_t}, \operatorname M_\beta ]
\cdots ] \overline\gamma .2. .
\end{equation*}
Recall that $ \beta = \operatorname P _{\mathcal U} b$ and the definition of $ \gamma $ 
in (\ref{e.gammaDefined}). 
The commutator is a linear combination of terms  
$
 \operatorname T [\beta \cdot \operatorname T' \overline\gamma ] \,,  
$
where $ \operatorname T$ and $ \operatorname T'$ are either 
the identity, or a product  in $ T_ {C_s}$, 
$ 1\le s\le t$. (Each $\operatorname  T_{C_s} $  must occur in either 
$ \operatorname T$ 
or  $ \operatorname T'$.)
In each case that $ \operatorname T' $ is a non-trivial 
product, we have  $ \operatorname T'\overline \gamma =0$.  
It then remains to consider the only term 
not of this type, namely 
\begin{equation*}
 \operatorname T_{C_1} \cdots \operatorname T_{C_t}
[ \beta \cdot \overline\gamma  ] \,.
\end{equation*}

Write $ \beta = \gamma + \beta '+\beta ''$, where the smaller 
cones from Lemma~\ref{l.coneSelection} enter again below. 
\begin{equation}\label{e.beta'} 
\begin{split}
\beta ' &= 
(\operatorname H _{C_1}\cdots \operatorname H _{C_t} 
- \operatorname T _{D_1}\cdots \operatorname T _{D_t}) \beta  \,,
\\
\beta ''&= (\operatorname I - \operatorname H _{C_1}\cdots \operatorname H _{C_t} ) \beta  \,.
\end{split}
\end{equation}
Note that (\ref{e.GAMMA}) provides information about $ \beta '$.  
We need to consider $ \operatorname T_{C_1} \cdots \operatorname T_{C_t} 
[\beta \cdot \overline \gamma  ]$, which is now divided  into three terms. 
They are 
\begin{equation*}
\begin{split}
\operatorname T_{C_1} \cdots \operatorname T_{C_t} 
[\beta \cdot \overline\gamma  ] 
&= 
\operatorname T_{C_1} \cdots \operatorname T_{C_t} 
[\beta '' \cdot \overline\gamma  ]
+ 
\operatorname T_{C_1} \cdots \operatorname T_{C_t} 
[\beta '\cdot \overline{\gamma }  ]
\\
& \qquad +\operatorname T_{C_1} \cdots \operatorname T_{C_t} 
[\gamma  \cdot \overline\gamma   ] \,. 
\end{split}
\end{equation*}

Now, $\overline \gamma $ and $ \beta ''$ are supported on the same product of halfspaces, 
which are complementary to the cones, thus
\begin{equation}\label{e.beta''<}
\operatorname T_{C_1} \cdots \operatorname T_{C_t} 
[\beta '' \cdot \overline\gamma  ]=0\,.
\end{equation}
For $ \beta '$ we do not attempt to find any cancellation, just relying on the 
favorable estimate from (\ref{e.1stDiff}). 
\begin{equation} \label{e.beta'<}
\norm 
\operatorname T_{C_1} \cdots \operatorname T_{C_t} 
[  \beta' \cdot \overline{\gamma }  ] .2. 
\le \norm  \beta' .4. \cdot \norm \gamma  .4. \lesssim \kappa \,.  
\end{equation}
The last term holds the essence of this component of the argument. 
By (\ref{e.2ndDiff}), 
\begin{equation} \label{e.PCCC} 
\begin{split}
\norm \operatorname T_{C_1} \cdots \operatorname T_{C_t} 
[ \gamma  \cdot \overline\gamma  ] .2.-\kappa 
&
\ge \norm 
\operatorname H _{C_1} \cdots \operatorname H _{C_t} 
[\overline {\gamma } \cdot \gamma  ] .2.
\\ 
& \gtrsim 
\norm  \overline {\gamma } \cdot \gamma .2.  
\\
& = \norm \gamma .4.^2 
\\
& \gtrsim  \NOrm \Bigl[ \sum _{ \vec \varepsilon \in \operatorname {Sig} _{\vec d} } 
\sum_{ R\in \mathcal U}
\frac {\abs{ \ip \gamma , w _{R} ^{\vec \varepsilon } ,} ^2 } {\abs{ R}} \mathbf 1 _{R} 
\Bigr] ^{1/2} 
.4. ^2  
\\
& \gtrsim  \NOrm \Bigl[ \sum _{ \vec \varepsilon \in \operatorname {Sig} _{\vec d} } 
\sum_{ R\in \mathcal U}
\frac {\abs{ \ip \gamma , w _{R} ^{\vec \varepsilon } ,} ^2 } {\abs{ R}} \mathbf 1 _{R} 
\Bigr] ^{1/2} 
.2. ^2 
\\
& \gtrsim 1\,. 
\end{split}
\end{equation}  
The second line follows as the Fourier transform of $ \overline {\gamma } \cdot \gamma $ 
is symmetric with respect to the half planes determined by the cones; 
the third line is obvious; the fourth line uses the Littlewood--Paley inequalities; 
and the fifth line uses the fact that the rectangles in $ \mathcal U$ are in a set of 
measure at most one. 
This completes the proof of (\ref{e.Obvious}).

The remainder of the paper is devoted to 
proving (\ref{e.Paraprod}) which is taken up in the next section.

\subsection*{The Proof of (\ref{e.Paraprod}) }\label{ss.paraproderrorterms}

Theorem~\ref{t.Ud} will let us conclude estimate (\ref{e.Paraprod}),
and the proof of Theorem \ref{t.lower}.  

We have observed that the commutator we are considering simplifies considerably 
when applied to the function $\overline \gamma $.  Letting $ \operatorname T 
= \operatorname T _{C_1} \cdots \operatorname T _{C_t}$, 
and $ \operatorname T' =  \operatorname T _{D_1} \cdots \operatorname T _{D_t}$
the estimate  we are to prove is 
\begin{equation*}
\norm \operatorname T  (\operatorname P _{\mathcal W} b \cdot \overline\gamma ).2. 
=\norm \operatorname T  (\operatorname P _{\mathcal W} b \cdot \overline {\operatorname T' \beta }).2. 
\lesssim K_J \delta _{-1} \,. 
\end{equation*}
It is critical to observe that the outermost $ \operatorname T$ imposes 
a cancellation condition similar to the one defining paraproducts.  
For $ R\in \mathcal U$ and $ R'\in \mathcal W$, and choices of signatures 
$ \vec\varepsilon ,\vec\varepsilon' \in \operatorname {Sig} _{\vec d}$, 
we have 
\begin{equation*}
\operatorname T ( w ^{\vec\varepsilon'}_{R'} \cdot \overline{ 
\operatorname T' w ^{\vec\varepsilon }_R } )=0
\qquad \textup{if  for  any $ 1\le s\le t$, $ \abs{ Q'_s }>64 \abs{ Q_s}$. } 
\end{equation*}
Recall that we defined $ \gamma =  {\operatorname T' \beta }$. 
$ \operatorname T$  and $ \operatorname T'$ are a convolution operators, so the Fourier support 
of $ \operatorname T w ^{\vec\varepsilon  } _{R}$ is contained in the Fourier support 
of $ w ^{\vec\varepsilon  } _{R}$.  Therefore, this property follows immediately from the 
properties of the Meyer wavelet.  

Using this observation, we see the estimate to be proved is 
\begin{align} \label{e.T<-1}
\NOrm 
\sum_{\vec\varepsilon,\vec\epsilon\in\textup{Sig}_{\vec d}}\sum _{(R,R')\in \mathcal A}
\overline{\ip b, w_R^{\vec\varepsilon},}\, 
{\ip b ,w_{R'}^{\vec\varepsilon'}, }\, \operatorname T 
( w_{R'}^{\vec\varepsilon}  \cdot \overline{ 
\operatorname T' w ^{\vec\varepsilon }_R })
.2. \lesssim 
\delta _{-1}\,, 
\\ \nonumber
\mathcal A \eqdef \{ (R,R')\mid 
R\subset U\,,\ R'\not\subset V\,, \abs{ Q_s'}\le 64 \abs{ Q_s}\,, 1\le s\le t\}\,.
\end{align}
Notice that in the $ L^2$ norm, we are free to remove the operator $ \operatorname T$, 
as it is a bounded operator on $ L^2 (\mathbb R ^{\vec d})$.

It is essential to observe that this last sum can be written as 
 a sum  of 
paraproducts, as in  Theorem~\ref{t.Ud}. 
The purpose of these next  definitions is to decompose the 
  collection $ \mathcal A$ into appropriate
parts to which Theorem~\ref{t.Ud} applies. 
For an integer $ n\ge 1$, take 
\begin{gather*}
\mathcal U _n \eqdef  \{R\in \mathcal U \mid 2 ^{n-1}\le 
\textup{Emb} (R)\le 2 ^{n} \}\,,
\\
\mathcal A _n\eqdef \{ (R,R')\mid 
R \in \mathcal U _{n}\,,\ R'\not\subset V\,, \abs{ Q_s'}\le 64 \abs{ Q_s}\,, 1\le s\le t\}\,.
\end{gather*}
Here, $ \textup{Emb}$ is the function supplied to us by Journ\'e's Lemma, 
Lemma~\ref{l.journed-1}. 
Hence, as a consequence we have 
\begin{equation*}
\norm \operatorname P _{\mathcal U_n} b . \textup{BMO}. \lesssim 2 ^{C n} \delta _{-1}\,,
\end{equation*}
where $ C$ is a large constant depending only on $ \vec d$.  Observe that 
for $ (R,R')\in \mathcal A_n$ we necessarily have $ 2 ^{n-1} R\cap R'=\emptyset$. 
In particular, the assumption (\ref{e.Localized}) will hold with $ A \simeq 2 ^{n}$. 
From this separation, and the rapid decay of the Meyer wavelet, we will gain 
an arbitrarily large power of $ 2 ^{-n}$.  Thus the presence of the term 
$ 2 ^{Cn}$ in this last estimate turns out not to be a concern for us.  

Our estimate below is a consequence of (\ref{e.localized}), after a further 
decomposition of the sum to account for the role of the location of the zeros, 
controlled by the set $ J$ in (\ref{e.localized}). 
\begin{equation}\label{e.claim1}
\NOrm 
\sum_{\vec\varepsilon,\vec\varepsilon'\in\textup{Sig}_{\vec d}}\sum _{(R,R')\in \mathcal A_n}
\overline{\ip b, w_R^{\vec\varepsilon},}\, 
{\ip b ,w_{R'}^{\vec\varepsilon'}, }\,  
( w_{R'}^{\vec\varepsilon}  \cdot \overline{ 
\operatorname T' w ^{\vec\varepsilon }_R }).2. \lesssim   2 ^{-n}\delta _{-1}\,, 
\qquad n\ge 1\,. 
\end{equation}
Summation in $ n\ge 1$ will then prove (\ref{e.T<-1}).

Specifically, let  $ J\subset \{1 ,\dotsc, t \}$, and let  integer $ \vec k \in \mathbb Z ^{t}$ 
satisfy 
\begin{equation*}
k_s=8\,, \ s\not\in J\,, \qquad -8\le k_s\le 8\,,\ s\in J\,.
\end{equation*}
Let $ \mathcal A _{n,J,\vec k}$ be a subset of $ \mathcal A _{n}$ given by 
\begin{equation*}
\mathcal A _{n,J,\vec k} \eqdef 
\{(R,R')\in \mathcal A_n \mid  \abs{ Q'_s}\le 2 ^{-8} \abs{ Q_s}\,,\ s\not\in J\,, \quad 
\abs{ Q'_s}=2 ^{k_s} \abs{ Q_s}\,,\ s\in J\}\,. 
\end{equation*}
For this collection, the estimate 
\begin{equation*}
\NOrm 
\sum_{\vec\varepsilon,\vec\varepsilon'\in\textup{Sig}_{\vec d}}\sum _{(R,R')\in \mathcal A 
_{n,J,\vec k}}
\overline{\ip b, w_R^{\vec\varepsilon},}\, 
{\ip b ,w_{R'}^{\vec\varepsilon'}, }\,  
( w_{R'}^{\vec\varepsilon}  \cdot \overline{ 
\operatorname T' w ^{\vec\varepsilon }_R }).2. \lesssim   2 ^{-n}\delta _{-1}
\end{equation*}
is then a consequence of (\ref{e.localized}).  This estimate is summed over 
$ J\subset \{1 ,\dotsc, n \}$ and $ \vec k$ to prove (\ref{e.claim1}).

\subsection{ The Lower Bound on  the Riesz Transforms}

\subsubsection*{Properties of Riesz Transforms}
  We need some special properties of Riesz transforms.  
Variants are in the paper of Coifman, Rochberg and Weiss \cite{MR54:843}.

\begin{proposition}\label{p.polynomials}
For each $s$, let $ \operatorname T_s$ be a polynomial in the Riesz transforms 
on $ \mathbb R ^{d_s}$.  Then, we have the inequality 
\begin{equation*}
\norm [\cdots [ \operatorname M_b, \operatorname T_1],\cdots ],\operatorname T_t] 
.2\to 2. 
\lesssim 
\sup _{\vec \jmath } \norm \operatorname C _{\vec \jmath }(b, \cdot ) .2\to 2.\,.
\end{equation*}
The implied constant depends upon the choice of polynomials $ \operatorname T_s$. 
\end{proposition}

This in fact follows from the elementary identity 
\begin{equation*}
[\operatorname M_b , \operatorname R _{j} \operatorname R_k] 
= [\operatorname M_b , \operatorname R _{j}] \operatorname R_k
+\operatorname R_j [\operatorname M_b , \operatorname R _{k}]\,.
\end{equation*}

An operator $ \operatorname T$ which is a polynomial in Riesz transforms is 
a convolution operator, with radial symbol.  Below, we will only describe the symbols 
that we are interested in.

\bigskip

The selection of the operators which are polynomials in Riesz transforms 
is hardly obvious, and we identify their properties in the following Lemma. 

\begin{lemma}\label{l.miracle} 
Given any $ 0<\eta <1$ and  any cone $ C$ in $ \mathbb R ^{d}$, 
there is an  operator $ \operatorname U_C$,
a linear combination of the identity and a polynomial in Riesz transforms on 
$ \mathbb R ^{d}$, with symbols $ \upsilon_C  $ such that 
\begin{equation} \label{e.UPsilon}
\begin{cases}
\abs{ \upsilon _C (\xi )-1}<\eta  & \xi \in C 
\\
\abs{ \upsilon _C  (\xi )}<\eta  & \xi \in -C \,. 
\end{cases}
\end{equation}
  Finally, we have the estimate 
\begin{equation}\label{e.pNorm}
\norm \operatorname U_C .p. \lesssim C_p\,,
\qquad 
1<p<\infty \,.
\end{equation}
The constant $ C_p$ is independent of the choice of the cone $ C$ 
and dimension $ d$.
\end{lemma}

\begin{proof}
This depends upon particular  properties of spherical harmonics and zonal 
polynomials.  We were aided by \cite {MR1805196} in our search for this proof. 

It suffices to prove the following.  For a choice of cone $ C$ in $ \mathbb R ^{d}$, with direction 
$ \xi _C$, and $ 0<\eta <1$, we can choose operator $\operatorname U $ 
with symbol $ \upsilon  $ which is  odd with respect to $ \xi _{C}$, 
$ \norm \upsilon . \infty . \le 2$, and 
\begin{equation} \label{e.Upsilon}
\begin{cases}
\abs{ \upsilon (\xi )+1} < \eta  & \xi \in -C \,,  
\\
\abs{ \upsilon (\xi )-1}<\eta  & \xi \in C\,.
\end{cases}
\end{equation}
Finally, $ \upsilon $ restricted to the unit sphere is a polynomial in 
$ \xi _j$ for $ (\xi _1,\xi _2,\dotsc, \xi _d)\in S ^{d-1}$.  We 
will see that the degree is at most $ \lesssim \eta  ^{-1}  \log 1/ \eta $.

Then, $ \operatorname U$ is in fact a polynomial in Riesz transforms.  Since 
$ \upsilon $ is odd, the method of rotations applies to provide us with an estimate 
of $ \norm \operatorname U.p. \le C_p$, where $ C_p$ is absolute for $ 1<p<\infty $. 
To get an operator with symbol as in our Lemma, we add the identity  operator to 
$ \operatorname U$. 

\medskip 

We obtain the symbol $ \upsilon $ by employing the Poisson kernel in the 
ball in $ \mathbb R ^{d}$, and as well an expansion of this kernel 
into zonal harmonics.  Let us recall the properties we need.  
The Poisson kernel in $ \mathbb R ^{d}$  is 
\begin{equation*}
P (x, \zeta  )=\frac {1- \abs{ x} ^2 } {\abs{ x-\zeta } ^{d}} \,,
\qquad 
\abs{ x}<1\,,\, \abs{ \zeta }=1 \,. 
\end{equation*}

A homogeneous harmonic polynomial  $ p$ of degree  $ m$ on $ \mathbb R ^{d}$ has 
the reproducing formula \cite[p. 97] {MR1805196} 
\begin{equation*}
p (x)= \abs{ x} ^{m} \int _{S ^{d-1}} p (\zeta ) Z_m (x, \zeta ) 
\; \sigma (d \zeta )\,. 
\end{equation*}
Here, $ \sigma $ denotes normalized Lebesgue measure on the unit 
sphere $ S ^{d-1}$ in $ \mathbb R ^{d}$. 
The polynomial $ Z_m (x, \zeta )$ is a \emph{zonal polynomial of degree $ m$}. 
It follows that the Poisson kernel admits an expansion in terms of these 
polynomials 
\begin{equation*}
P (x, \xi ) = \sum _{m=0} ^{\infty } Z_m (x, \zeta )\,. 
\end{equation*}
This series is absolutely convergent sum in the interior of the unit ball, 
thanks to the elementary estimate 
\begin{equation}\label{e.Z<}
\abs{ Z _{m} (x,\zeta )} \lesssim m ^{d-2} \abs{ x} ^{m}\,. 
\end{equation}

It is a basic property of the zonal polynomials that they are only a function of 
$ \abs{ x} $ and $ x \cdot \zeta $.  Indeed, they can be expanded as 
\begin{equation*}
Z _{m} (x,\zeta )=\sum _{k=0} ^{[m/2]} c _{k,m} (x \cdot \zeta ) ^{m-2k} \abs{ x} ^{2k}\,. 
\end{equation*}
Here, $ c _{k,m}$ are known real coefficients. 
In particular, for a function $ \widetilde \upsilon $ on the unit 
sphere that is odd, the new function  
\begin{equation*}
\int \widetilde \upsilon (\zeta ) Z_m (x,\zeta ) \; d \sigma (\zeta )
\end{equation*}
is also odd, for $ \abs{ x}$ held fixed.

\smallskip 

To construct our operator $ \operatorname U$, via its symbol $ \upsilon $, 
recall that the operator $ \operatorname U$ is associated to a 
cone $ C$ with direction $ \xi _C$.  On the cone $ C$ and the opposite 
cone $ -C$ we require rather precise information about the symbol $ \upsilon $.
Outside of these cones we only require an absolute upper bound 
on $ \upsilon $.  Hence, we have some freedom 
in taking an initial approximate to the symbol $ \upsilon $.  In what 
follows, we concentrate on defining the symbol on the sphere $ S ^{d-1}$. 

Take as an initial approximate $ \widetilde {\widetilde \upsilon} (\xi )= 
\operatorname {sign} (\xi \cdot \xi _{C})$.  For a small constant $ c$, 
consider the function 
\begin{equation*}
\widetilde \upsilon (\theta ) 
\eqdef 
\int    \widetilde {\widetilde \upsilon} (\zeta ) P ((1-c \eta ) \theta ,\zeta )
\; d \sigma(\zeta) \,. 
\end{equation*}
This function will be non-negative, 
odd, bounded in absolute value 
by $ 1$, and satisfy (\ref{e.Upsilon}).   It is not however a polynomial 
in spherical harmonics.  

But each function 
\begin{equation*}
\upsilon _m (\theta ) \eqdef \int    \widetilde {\widetilde \upsilon} (\zeta ) Z_m ((1-c \eta ) \theta ,\zeta )
\; d\sigma(\zeta) 
\end{equation*}
is also odd, as we have already noted.  By 
(\ref{e.Z<}), we have the estimate 
\begin{equation*}
\NOrm \sum _{m=m_0} ^{\infty } \upsilon _m . L ^{\infty } (S ^{d-1}). 
\le \eta /4\,, 
\qquad m_0=C (\log 1/ \eta )/\eta 
\end{equation*}
where $ C$ depends upon the dimension $ d$.  Therefore, the function 
$ \upsilon $ we need can be taken to be 
\begin{equation*}
\sum _{m=0} ^{m_0} \upsilon _m \,. 
\end{equation*}
Our proof is complete. 

\end{proof}
 
\subsubsection*{The Selection of a Test Function}

We continue to assume that the symbol $ b$ satisfies $ \norm b .\textup{BMO}.=1$ 
while $  \norm b .\textup{BMO} _{-1}.<\delta _{-1}$.  We follow 
many of the initial stages of the proof of the lower bound on the Cone norm.
We choose cones $ D_s\subset C_s$, and cone operators $ \operatorname T _{D_s}, \operatorname T _{C_s}$, 
for $ 1\le s \le t$ just as in Lemma~\ref{l.coneSelection}.

We continue to use the notations $ \mathcal U$, $ \mathcal V$ and $ \mathcal W$, thus 
 $ \abs{\operatorname {sh} (\mathcal U)} \simeq 1$ and 
\begin{equation*}
\sum _{R\in \mathcal U} \abs{ \ip b, w_R,} ^2 \simeq \abs{ \operatorname {sh} (\mathcal U)}. 
\end{equation*}
$ \mathcal V$ and $ \mathcal W$ are defined as in (\ref{e.V}), and 
 $ \beta =\operatorname P _{\mathcal U} b$. 
 As before, we set $  \gamma =\operatorname T _{D_1}\cdots \operatorname T _{D_t}
 \beta $. 

For $ 0<\eta <1$ to be chosen, apply Lemma~\ref{l.miracle}, obtaining 
operators $ \operatorname T_s$ which are a linear combinations of the
identity and polynomials in Riesz transforms on $ \mathbb R ^{d_s}$ which 
approximate the projection operator $\operatorname  P _{C_s}$ in the sense of that 
Lemma.   Let us see that we have the estimate 
\begin{equation}\label{e.TTT}
\norm [\operatorname T _{1}, \cdots [ 
\operatorname T _{t},\operatorname M_\beta ] \cdots ] \overline\gamma .2. \gtrsim 1\,.
\end{equation}
The commutator is a linear combination of $ 2^t$ terms of the form 
\begin{equation*}
\operatorname T[\beta  \operatorname T' \overline\gamma ]
\end{equation*}
where $ \operatorname T$ and $ \operatorname T'$ are either the identity or 
a composition of the operators $\operatorname  T_{s}$.  
If $ \operatorname T'$ is not the identity, it follows that the symbol of $ \operatorname T'$
is at most $ \eta$ on the Fourier support of $ \overline\gamma $.  Therefore, we can estimate 
\begin{align*}
\norm \operatorname T[\beta  \operatorname T' \overline\gamma ] .2. 
\lesssim   \norm \beta  \operatorname T' \overline\gamma .2. 
 \lesssim  \norm \beta .4. \, \norm \operatorname T' \overline\gamma .4.
 \lesssim \eta ^{1/3}. 
\end{align*}
This last estimate follows from $ \norm \operatorname T' \overline\gamma .2. \lesssim \eta $, 
while $  \norm \operatorname T' \overline\gamma .8. \lesssim 1$.  This point is imperative, 
and follows from the uniform $ L^p$ bounds we obtain from Lemma~\ref{l.miracle}.

This leaves the term $ \operatorname T_1 \cdots \operatorname T_t \beta \overline\gamma $. 
But, for a sufficiently small choice of $ \eta $, 
we are free to use the same argument as in (\ref{e.PCCC}). This proves (\ref{e.TTT}). 

\smallskip

It then follows from Proposition~\ref{p.polynomials} that for some choice 
of Riesz transforms $ \operatorname R _{j_s}$ on $ \mathbb R ^{d_s}$  
we have 
\begin{equation}\label{e.RRR}
\begin{split}
\norm  \operatorname C( \beta , \overline\gamma ). 2.& \gtrsim 1\,, 
\\
\textnormal{ where }\operatorname C (f,g) \eqdef 
 [\operatorname R _{j_1}, &\cdots [ 
\operatorname R _{j_t}, \operatorname M_ {f}  ] \cdots ]g\,.
\end{split}
\end{equation}

Now, it also follows that 
\begin{equation}\label{e.JOURNE}
\norm \operatorname C(\operatorname P _{\mathcal V} b,\overline\gamma ) .2. 
\lesssim \delta _{J} ^{1/4}\,. 
\end{equation}
Indeed, we can appeal to the same argument as used to prove (\ref{e.Journe}).  

Finally, we claim that 
\begin{equation}\label{e.PARA}
\norm \operatorname C(\operatorname P _{\mathcal W} b,\overline\gamma ) .2. 
\lesssim \delta _{-1}\,. 
\end{equation}
This estimate requires the same argument as for  (\ref{e.Paraprod}), 
plus an additional estimate; the 
details are below. 
The three inequalities (\ref{e.RRR}), (\ref{e.JOURNE}) and (\ref{e.PARA}) 
are then combined in in the same manner as in the proof of the lower bound 
on the Cone norm to complete  the proof.

\subsubsection*{Proof of (\ref{e.PARA})}

The different quantitative estimates we have for paraproducts are 
brought to bear on this estimate.  
First, we expand the expression $\operatorname C
(\operatorname P_{\mathcal W}b,\gamma )$ into the sum 
of commutators on different pairs of wavelets.  This sum is further written as 
$ D_1+D_2$, where we define $ D_1$ explicitly here. 
\begin{equation}\label{e.notComplete}
\begin{split}
D_1 &\eqdef
\sum_{\vec\varepsilon,\vec\epsilon\in\textup{Sig}_{\vec d}}
\sum _{(R,R')\in \mathcal A} 
\overline{\ip \gamma , w_R^{\vec\varepsilon},}\, \ip b ,w_{R'}^{\vec\epsilon}, 
\, \operatorname C(w_{R'}^{\vec\epsilon}, \overline{ w_R^{\vec\varepsilon'}})
\\
\mathcal A &\eqdef \{ (R,R')\mid 
R\in \mathcal U\,,\ R'\not\subset V\,, \abs{ Q_s'}\le 64 \abs{ Q_s}\,, 1\le s\le t\}\,.
\end{split}
\end{equation}
This is the part of the commutator that most closely resembles the part 
arising from the commutator arising from the Cone operators.

It is essential to observe that this last sum can be written as a finite sum of compositions of Riesz transforms and the 
``paraproducts'' in Theorem~\ref{t.Ud}, 
and in particular  the more technical estimate (\ref{e.localized}),
applied to the functions $\operatorname P_{\mathcal U}b $ and $ \operatorname P_{\mathcal W}b$.  We also comment that the Riesz transforms applied to the wavelet element $w_R^{\vec\varepsilon}$ do not substantially change the localization properties of the wavelet, and thus the Riesz transforms do not spoil the estimates that appear in Theorem~\ref{t.Ud}.  
This sum varies of choices of $\vec k$ with $ \norm\vec k.\infty.\leq 8$, and
arbitrary $ J\subset \{
1,\dotsc,t\}$. (The subset $ J$ consists of those coordinates  $ s$ for which 
$ \abs{ Q_s} =2 ^{k_s} \abs{ Q'_s}$.)

We will need to decompose the collection $ \mathcal A$ into appropriate
parts to which this estimate applies.  That is the purpose of this definition.
For an integer $ n\ge 1$, take 
\begin{equation*}
\gamma  _{n} \eqdef \sum_{\vec\varepsilon\in\textup{Sig}_{\vec d}}\,\sum _{\substack{R\subset U\\ 
2 ^{n-1}\le 
\textup{Emb} (R)\le 2 ^{n}}} \ip \gamma , w_R^{\vec\varepsilon}, w_R^{\vec\varepsilon}
\end{equation*}
We claim that 
\begin{equation}\label{e.claimRiesz}
\norm \operatorname  C(\operatorname P_{\mathcal W}b,\overline{\gamma _{n}}).2. \lesssim 2 ^{-n} \delta _{-1}\,. 
\end{equation}
It follows from Lemma~\ref{l.journed-1} that we have the estimate 
\begin{equation}\label{e.alpha<}
\norm \gamma _{n}. \textup{BMO}(\mathbb R^{\vec d}).
\lesssim 2 ^{2t n} \delta _{-1}\,,
\end{equation}
indeed, this is the point of this definition. From other parts of the expansion of the Riesz commutator  we need to find some decay in $ n$.

Nevertheless, from this estimate and the  upper bound on Riesz commutator norms, we have 
the estimate 
\begin{equation*}
\norm 
\operatorname  C(\operatorname P_{\mathcal W}b,\overline{\gamma_{n}}).2. \lesssim 
\norm b. \textup{BMO}(\mathbb R^{\vec d}). \norm \gamma_{n}.2.
\lesssim 2 ^{2t n} \delta _{-1}.
\end{equation*}
We use this estimate for $ n<20$, say.

For $ n\ge20$, $ R\in \mathcal  U$ with $2 ^{n-1}\le 
\textup{Emb} (R)\le 2 ^{n} $, and rectangle $ R'$ with $ (R,R')\in \mathcal A$, 
it follows that we must have $ 2 ^{n-9}R\cap R'=\emptyset$. 
That is, (\ref{e.Localized}) is satisfied with the value of $ A$ in that 
display being $ A\simeq 2 ^{n}$ for $ n\ge20$.  Thus, we conclude that 
\begin{equation*}
\norm 
\operatorname  C(\operatorname P_{\mathcal W}b ,\overline{\gamma_{n}}).2. \lesssim 
2 ^{-50n} \delta _{-1},\qquad n\ge 20\,. 
\end{equation*}
This completes our proof of (\ref{e.claimRiesz}), and the proof  estimate (\ref{e.Paraprod}) .

\bigskip 

It remains to estimate the term $ D_2$.  The principal tool here is the estimate 
for Riesz commutators given in Lemma~\ref{l.quantReisz}, and in order to 
apply this lemma, as well as take advantage of our remaining freedom 
to select the $ \delta _{-1}$ norm, we need a sophisticated decomposition of the sum that 
controls $ D_2$.  That is the point of these next definitions.

Let $ m$ be an integer.  For a non empty subset $ J\subset \{1,\dotsc,t\}$, 
and choices of integers $ \vec a= (a_j)_{j\in J}$ with  $a_j\ge 7$, 
we define 
\begin{equation}\label{e.D2} 
\begin{split}
D (m,J,\vec a) & \eqdef 
\sum_{\vec\varepsilon,\vec\varepsilon'\in\textup{Sig}_{\vec d}}
\sum _{(R,R')\in \mathcal A (m,J,\vec a)} 
\overline{\ip \gamma , w_R^{\vec\varepsilon},}\, \ip b ,w_{R'}^{\vec\varepsilon'}, 
\, \operatorname C(w_{R'}^{\vec\varepsilon'},  \overline{w_R^{\vec\varepsilon}})
\\
\mathcal A (m,J,\vec a) &\eqdef \{ (R,R')\mid 
R\in \mathcal U\,,\ R'\not\subset V\,;\, \abs{ Q_s'}\le 64 \abs{ Q_s}\,, s\not\in J 
\,; \ 
\\
& \qquad 
\abs{ Q_s'}=2 ^{a} \abs{ Q_s}\,,\ s\in J\,;\,  
  2 ^{m}\le \operatorname {Emb} (R) \le 2 ^{m+1}
\}\,.
\end{split}
\end{equation}
With this definitions, we will have 
\begin{equation*}
D_2=\sum _{m, J,\vec a} D (m,J,\vec a) \,.
\end{equation*}
 The estimate below holds. 
\begin{equation}\label{e.D<}
\norm D (m,J,\vec a) .2. \lesssim  2 ^{-m- \sum _{s} a_s} \delta _{-1}\,. 
\end{equation}
This is summable over the parameters $ m, J, \vec a$, and so completes the proof 
of the estimate for $ D_2$. 

\smallskip 

Recall that we have the essential consequence of Journ\'e's Lemma.  
\begin{equation}\label{e.Gm}
\NOrm \sum _{\substack{R\in \mathcal U\\
2 ^{m}\le \operatorname {Emb} (R) \le 2 ^{m+1}}} 
\sum _{\varepsilon \in \operatorname {Sig} _{\vec d}}
\ip \gamma , w ^{\varepsilon } _{R}, w ^{\varepsilon } _{R} . \textup{BMO}. 
\lesssim 2 ^{C _{\vec d}\, m} \delta _{-1}\,. 
\end{equation}

The first subcase occurs when we have 
\begin{equation}\label{e.s1}
\max _{s\in J} a_s\le \tfrac m4 \,. 
\end{equation}
It follows from the definition of embeddedness 
that $  Q_{s'} \cap 2 ^{m/4} Q_s=\emptyset $ 
for all $ (R,R')\in \mathcal A (m,J,\vec a)$.  Therefore, the function 
\begin{equation} \label{e.s11}
2 ^{-N m}
\sqrt {\abs{ Q _{s'}}}\operatorname C\,(w_{R'}^{\vec\varepsilon'}, \overline{ w_R^{\vec\varepsilon}})
\end{equation}
is adapted to $ Q _{s}$, where $ N\ge 1$ can be taken arbitrarily.  
This sum can be understood as a paraproduct, with symbol given by 
$ \operatorname P _{\mathcal W}b$, applied to $ \gamma $. 
Zeros fall on $ \gamma $ for those coordinates $ s\in J$.  
Using (\ref{e.Gm}) and the estimate (\ref{e.s1}), we see that 
(\ref{e.D<}) holds in this case. 

\smallskip 

The second case is when (\ref{e.s1}) fails in any coordinate, say $ s_0\in J$. 
In this instance, we see that the sum we are considering in that coordinate 
is of the type considered in Lemma~\ref{l.quantReisz}.  That is, the sum is 
an operator whose paraproduct norm, as defined in (\ref{e.paraNorm}) is 
at most $ 2 ^{-N a _{s_0}}$, for arbitrarily large $ N$.  
In all other coordinates, the sum is an operator with paraproduct norm 
at most a constant.  The tensor product of paraproducts is a bounded operator, 
therefore in this case, we have 
\begin{equation*}
\norm D (m,J,\vec a).2. \lesssim 2 ^{- N \sum a_s}  2 ^{C _{\vec d} m} \delta _{-1}
\lesssim 2 ^{-2 \sum a_s} \delta _{-1}\,. 
\end{equation*}
Here of course, we again rely upon (\ref{e.Gm}). 
This completes the proof of (\ref{e.D<}).   This in turn completes the proof of the lower bound on the norm of multiparameter Riesz commutators. 

  
\end{document}